\documentclass[12pt]{article}
\usepackage{graphicx}
\usepackage{amssymb}
\usepackage{pb-diagram}
\graphicspath{c:/aboutlatex/texte/mathe/}
\newcommand{\agd}{abstract graph diagram}
\newcommand{\cc}{c.c.}
\newcommand{\kref}[1]{(\ref{#1})}
\newcommand{\RR}[1]{R\left(#1\right)}

\newtheorem{theorem}{Theorem}[section]

\newtheorem{proposition}[theorem]{Proposition}
\newtheorem{definition}[theorem]{Definition}
\newtheorem{example}[theorem]{Example}
\newtheorem{remark}[theorem]{Remark}

\newenvironment{proof}{{\bf Proof.}}{\rightline{$\Box$}}
\newenvironment{proofX}{{\bf Proof.}}{}
\newenvironment{proofofmaintheorem}{{\noindent\bf Proof of Theorem.}}{}
\begin{document}
\newcommand{\keywords}[1]{{\it\small Keywords: }{\small #1}}
\newcommand{\ccode}[1]{{\it AMS classification: }{\small #1}}
\newcommand{\tgd}{twisted graph diagram}
\newcommand{\R}{\mbox{$\mathbb{R}$}}
\newcommand{\orient}[1]{\mbox{orien\-tation}\, #1}
\newcommand{\signt}[1]{\mbox{sign}\left(#1\right)}

\newcommand{\doubleLoop}{
\hspace*{1pt}\begin{minipage}{10pt}
\begin{picture}(10,10)
\put(3,5){\circle{10}}
\put(7,5){\circle{10}}
\put(5,1){\circle*{3}}
\end{picture}
\end{minipage}\hspace*{1.5pt}
}

\newcommand{\ThetaKlein}{
\hspace*{1pt}\begin{minipage}{10pt}
\begin{picture}(10,10)
\put(5,5){\circle{10}}
\put(0,5){\line(1,0){10}}
\put(0,5){\circle*{3}}
\put(10,5){\circle*{3}}
\end{picture}
\end{minipage}\hspace*{1.5pt}
}

\newcommand{\ThetaOhne}{
\hspace*{1pt}\begin{minipage}{10pt}
\begin{picture}(10,10)
\put(5,5){\circle{10}}
\put(0,5){\circle*{3}}
\put(10,5){\circle*{3}}
\end{picture}
\end{minipage}\hspace*{1.5pt}
}

\newcommand{\twoLoop}{
\hspace*{5pt}\begin{minipage}{15pt}
\begin{picture}(20,10)
\put(0,5){\circle{10}}
\put(10,5){\circle{10}}
\put(5,5){\circle*{3}}
\end{picture}
\end{minipage}\hspace*{1.5pt}
}

\newcommand{\ThetaV}{
\hspace*{1pt}\begin{minipage}{10pt}
\begin{picture}(10,10)
\put(5,5){\circle{10}}
\put(0,5){\circle*{3}}
\put(10,5){\circle*{3}}
\qbezier(0,5)(-5,15)(10,5)
\end{picture}
\end{minipage}\hspace*{1.5pt}
}

\newcommand{\eight}{
\hspace*{1pt}\begin{minipage}{10pt}
\begin{picture}(10,10)
\put(0,5){\circle*{3}}
\put(10,5){\circle*{3}}
\qbezier(0,5)(5,15)(10,5)
\qbezier(0,5)(-5,15)(10,5)
\end{picture}
\end{minipage}\hspace*{1.5pt}
}

\newcommand{\twovertex}{
\hspace*{1pt}\begin{minipage}{10pt}
\begin{picture}(10,10)
\put(1,5){\circle*{3}}
\put(9,5){\circle*{3}}
\end{picture}
\end{minipage}\hspace*{1.5pt}
}

\newcommand{\linie}{
\hspace*{1pt}\begin{minipage}{10pt}
\begin{picture}(10,10)
\put(1,5){\circle*{3}}
\put(9,5){\circle*{3}}
\put(1,5){\line(1,0){8}}
\end{picture}
\end{minipage}\hspace*{1.5pt}
}

\newcommand{\Tiv}{\hspace*{1pt}
\begin{minipage}{30pt}
\begin{picture}(30,10)
\qbezier(15,5)(22.5,15)(30,0)
\qbezier(15,5)(7.5,-5)(0,10)
\qbezier(16.5,3.5)(22.5,-5)(30,10)
\qbezier(13.5,6.5)(7.5,15)(0,0)
\qbezier(0,3)(0,3)(2,1)
\qbezier(30,3)(30,3)(28,1)
\qbezier(0,7)(0,7)(2,9)
\qbezier(30,7)(30,7)(28,9)
\end{picture}
\end{minipage}
\hspace*{1.5pt}}
\newcommand{\TivVertex}{\hspace*{1pt}
\begin{minipage}{30pt}
\begin{picture}(30,10)
\qbezier(15,5)(22.5,15)(30,0)
\qbezier(15,5)(7.5,-5)(0,10)
\qbezier(15,5)(22.5,-5)(30,10)
\qbezier(15,5)(7.5,15)(0,0)
\qbezier(0,3)(0,3)(2,1)
\qbezier(30,3)(30,3)(28,1)
\qbezier(0,7)(0,7)(2,9)
\qbezier(30,7)(30,7)(28,9)
\put(15,5){\circle*{3}}
\end{picture}
\end{minipage}
\hspace*{1.5pt}}
\newcommand{\TivInf}{\hspace*{1pt}
\begin{minipage}{20pt}
\begin{picture}(20,10)
\qbezier(11,5)(12.5,15)(20,0)
\qbezier(9,5)(7.5,-5)(0,10)
\qbezier(11,5)(12.5,-5)(20,10)
\qbezier(9,5)(7.5,15)(0,0)
\qbezier(0,3)(0,3)(2,1)
\qbezier(20,3)(20,3)(18,1)
\qbezier(0,7)(0,7)(2,9)
\qbezier(20,7)(20,7)(18,9)
\end{picture}
\end{minipage}
\hspace*{1.5pt}}
\newcommand{\TivNil}{\hspace*{1pt}
\begin{minipage}{20pt}
\begin{picture}(20,10)
\qbezier(0,0)(10,18)(20,0)
\qbezier(0,10)(10,-8)(20,10)
\qbezier(0,3)(0,3)(2,1)
\qbezier(20,3)(20,3)(18,1)
\qbezier(0,7)(0,7)(2,9)
\qbezier(20,7)(20,7)(18,9)
\end{picture}
\end{minipage}
\hspace*{1.5pt}}

\newcommand{\infcrossPa}{
\hspace*{1pt}
\begin{minipage}{10pt}
\begin{picture}(10,10)
\qbezier(0,0)(5,6)(10,0)
\qbezier(0,10)(5,4)(10,10)
\qbezier(-1,8)(-1,8)(11,8)
\end{picture}
\end{minipage}
\hspace*{1.5pt} 
}
\newcommand{\infcrossPb}{
\hspace*{1pt}
\begin{minipage}{10pt}
\begin{picture}(10,10)
\qbezier(0,0)(5,6)(10,0)
\qbezier(0,10)(5,4)(10,10)
\qbezier(-1,2)(-1,2)(11,2)
\end{picture}
\end{minipage}
\hspace*{1.5pt} 
}
\newcommand{\nilcrossPa}{
\hspace*{1pt}
\begin{minipage}{10pt}
\begin{picture}(10,10)
\qbezier(0,0)(6,5)(0,10)
\qbezier(10,0)(4,5)(10,10)
\qbezier(-1,8)(-1,8)(11,8)
\end{picture}
\end{minipage}
\hspace*{1.5pt}
}
\newcommand{\nilcrossPb}{
\hspace*{1pt}
\begin{minipage}{10pt}
\begin{picture}(10,10)
\qbezier(0,0)(6,5)(0,10)
\qbezier(10,0)(4,5)(10,10)
\qbezier(-1,2)(-1,2)(11,2)
\end{picture}
\end{minipage}
\hspace*{1.5pt}
}
\newcommand{\vertexPa}{\hspace*{1pt}
\begin{minipage}{10pt}
\begin{picture}(10,10)
\qbezier(0,0)(0,0)(10,10)
\qbezier(0,10)(0,10)(10,0)
\qbezier(-1,8)(-1,8)(11,8)
\put(5,5){\circle*{3}}
\end{picture}
\end{minipage}
\hspace*{1.5pt}}
\newcommand{\vertexPb}{\hspace*{1pt}
\begin{minipage}{10pt}
\begin{picture}(10,10)
\qbezier(0,0)(0,0)(10,10)
\qbezier(0,10)(0,10)(10,0)
\qbezier(-1,2)(-1,2)(11,2)
\put(5,5){\circle*{3}}
\end{picture}
\end{minipage}
\hspace*{1.5pt}}
\newcommand{\VvirtL}{\hspace*{1pt}
\begin{minipage}{10pt}
\begin{picture}(10,10)
\qbezier(0,0)(0,0)(10,10)
\qbezier(10,0)(10,0)(6.5,3.5)
\qbezier(0,10)(0,10)(3.5,6.5)
\qbezier(-1,8)(0,8)(11,8)
\end{picture}
\end{minipage}
\hspace*{1.5pt}}
\newcommand{\VvirtR}{\hspace*{1pt}
\begin{minipage}{10pt}
\begin{picture}(10,10)
\qbezier(0,0)(0,0)(10,10)
\qbezier(10,0)(10,0)(6.5,3.5)
\qbezier(0,10)(0,10)(3.5,6.5)
\qbezier(-1,2)(0,2)(11,2)
\end{picture}
\end{minipage}
\hspace*{1.5pt}}
\newcommand{\deleteU}{\hspace*{1pt}
\begin{minipage}{10pt}
\begin{picture}(10,6)
\qbezier(2,3)(2,3)(0,6)
\qbezier(2,3)(2,3)(0,0)
\qbezier(8,3)(8,3)(10,6)
\qbezier(8,3)(8,3)(10,0)
\put(2,3){\circle*{2}}
\put(8,3){\circle*{2}}
\end{picture}
\end{minipage}
\hspace*{1.5pt}}
\newcommand{\vertexU}{\hspace*{1pt}
\begin{minipage}{6pt}
\begin{picture}(6,6)
\qbezier(0,0)(0,0)(6,6)
\qbezier(0,6)(0,6)(6,0)
\put(3,3){\circle*{2}}
\end{picture}
\end{minipage}
\hspace*{1.5pt}}
\newcommand{\edgeU}{\hspace*{1pt}
\begin{minipage}{10pt}
\begin{picture}(10,6)
\qbezier(2,3)(2,3)(8,3)
\qbezier(2,3)(2,3)(0,6)
\qbezier(2,3)(2,3)(0,0)
\qbezier(8,3)(8,3)(10,6)
\qbezier(8,3)(8,3)(10,0)
\put(2,3){\circle*{2}}
\put(8,3){\circle*{2}}
\end{picture}
\end{minipage}
\hspace*{1.5pt}}
\newcommand{\edge}{\hspace*{1pt}
\begin{minipage}{16pt}
\begin{picture}(16,10)
\qbezier(3,5)(3,5)(13,5)
\qbezier(3,5)(3,5)(0,10)
\qbezier(3,5)(3,5)(0,0)
\qbezier(13,5)(13,5)(16,10)
\qbezier(13,5)(13,5)(16,0)
\put(3,5){\circle*{3}}
\put(13,5){\circle*{3}}
\end{picture}
\end{minipage}
\hspace*{1.5pt}}
\newcommand{\delete}{\hspace*{1pt}
\begin{minipage}{16pt}
\begin{picture}(16,10)
\qbezier(3,5)(3,5)(0,10)
\qbezier(3,5)(3,5)(0,0)
\qbezier(13,5)(13,5)(16,10)
\qbezier(13,5)(13,5)(16,0)
\put(3,5){\circle*{3}}
\put(13,5){\circle*{3}}
\end{picture}
\end{minipage}
\hspace*{1.5pt}}
\newcommand{\vertexO}{\hspace*{1pt}
\begin{minipage}{20pt}
\begin{picture}(20,20)
\qbezier(0,0)(0,0)(20,20)
\qbezier(0,20)(0,20)(20,0)
\put(10,10){\circle*{4}}
\end{picture}
\end{minipage}
\hspace*{1.5pt}}
\newcommand{\vertex}{\hspace*{1pt}
\begin{minipage}{10pt}
\begin{picture}(10,10)
\qbezier(0,0)(0,0)(10,10)
\qbezier(0,10)(0,10)(10,0)
\put(5,5){\circle*{3}}
\end{picture}
\end{minipage}
\hspace*{1.5pt}}
\newcommand{\poscross}{\hspace*{1pt}
\begin{minipage}{10pt}
\begin{picture}(10,10)
\qbezier(0,0)(0,0)(10,10)
\qbezier(10,0)(10,0)(6.5,3.5)
\qbezier(0,10)(0,10)(3.5,6.5)
\end{picture}
\end{minipage}
\hspace*{1.5pt}}
\newcommand{\poscrossO}{\hspace*{1pt}
\begin{minipage}{20pt}
\begin{picture}(20,20)
\qbezier(0,0)(0,0)(20,20)
\qbezier(20,0)(20,0)(13,7)
\qbezier(0,20)(0,20)(7,13)
\end{picture}
\end{minipage}
\hspace*{1.5pt}}
\newcommand{\poscrossU}{\hspace*{1pt}
\begin{minipage}{6pt}
\begin{picture}(6,6)
\qbezier(0,0)(0,0)(6,6)
\qbezier(6,0)(6,0)(4.5,1.5)
\qbezier(0,6)(0,6)(1.5,4.5)
\end{picture}
\end{minipage}
\hspace*{1.5pt}}
\newcommand{\negcross}{
\hspace*{1pt}
\begin{minipage}{10pt}
\begin{picture}(10,10)
\qbezier(0,10)(0,10)(10,0)
\qbezier(0,0)(0,0)(3.5,3.5)
\qbezier(6.5,6.5)(10,10)(10,10)
\end{picture}
\end{minipage}
\hspace*{1.5pt}
}
\newcommand{\negcrossO}{
\hspace*{1pt}
\begin{minipage}{20pt}
\begin{picture}(20,20)
\qbezier(0,20)(0,20)(20,0)
\qbezier(0,0)(0,0)(7,7)
\qbezier(13,13)(20,20)(20,20)
\end{picture}
\end{minipage}
\hspace*{1.5pt}
}
\newcommand{\negcrossU}{
\hspace*{1pt}
\begin{minipage}{6pt}
\begin{picture}(6,6)
\qbezier(0,6)(0,6)(6,0)
\qbezier(0,0)(0,0)(1.5,1.5)
\qbezier(4.5,4.5)(6,6)(6,6)
\end{picture}
\end{minipage}
\hspace*{1.5pt}
}
\newcommand{\nilcross}{
\hspace*{1pt}
\begin{minipage}{10pt}
\begin{picture}(10,10)
\qbezier(0,0)(6,5)(0,10)
\qbezier(10,0)(4,5)(10,10)
\end{picture}
\end{minipage}
\hspace*{1.5pt}
}
\newcommand{\nilcrossO}{
\hspace*{1pt}
\begin{minipage}{20pt}
\begin{picture}(20,20)
\qbezier(0,0)(12,10)(0,20)
\qbezier(20,0)(8,10)(20,20)
\end{picture}
\end{minipage}
\hspace*{1.5pt}
}
\newcommand{\infcross}{
\hspace*{1pt}
\begin{minipage}{10pt}
\begin{picture}(10,10)
\qbezier(0,0)(5,6)(10,0)
\qbezier(0,10)(5,4)(10,10)
\end{picture}
\end{minipage}
\hspace*{1.5pt} 
}
\newcommand{\infcrossO}{
\hspace*{1pt}
\begin{minipage}{20pt}
\begin{picture}(20,20)
\qbezier(0,0)(10,12)(20,0)
\qbezier(0,20)(10,8)(20,20)
\end{picture}
\end{minipage}
\hspace*{1.5pt} 
}
\newcommand{\infcrossU}{
\hspace*{1pt}
\begin{minipage}{6pt}
\begin{picture}(6,6)
\qbezier(0,0)(3,3)(6,0)
\qbezier(0,6)(3,3)(6,6)
\end{picture}
\end{minipage}
\hspace*{1.5pt} 
}
\newcommand{\nilcrossU}{
\hspace*{1pt}
\begin{minipage}{6pt}
\begin{picture}(6,6)
\qbezier(0,0)(3,3)(0,6)
\qbezier(6,0)(3,3)(6,6)
\end{picture}
\end{minipage}
\hspace*{1.5pt}
}

\newcommand{\loopU}{\hspace*{0.5pt}\begin{minipage}{8pt}
\begin{picture}(8,8)
\put(4,4){\circle{8}}
\put(4,0){\circle*{3}}
\end{picture}
\end{minipage}\hspace*{1pt}
}

\newcommand{\circbarU}{
\hspace*{0.5pt}\begin{minipage}{8pt}
\begin{picture}(8,8)
\put(4,4){\circle{8}}
\put(4,6){\line(0,0){4}}
\end{picture}
\end{minipage}\hspace*{1pt}
}
\newcommand{\circU}{
\hspace*{0.5pt}\begin{minipage}{8pt}
\begin{picture}(8,8)
\put(4,4){\circle{8}}
\end{picture}
\end{minipage}\hspace*{1pt}
}

\newcommand{\Loop}{
\hspace*{1pt}\begin{minipage}{10pt}
\begin{picture}(10,10)
\put(5,5){\circle{10}}
\put(5,0){\circle*{3}}
\end{picture}
\end{minipage}\hspace*{1.5pt}
}

\newcommand{\loopbar}{
\hspace*{1pt}\begin{minipage}{10pt}
\begin{picture}(10,10)
\put(5,5){\circle{10}}
\put(5,7.5){\line(0,0){5}}
\put(5,0){\circle*{3}}
\end{picture}
\end{minipage}\hspace*{1.5pt}
}

\newcommand{\circbar}{
\hspace*{1pt}\begin{minipage}{10pt}
\begin{picture}(10,10)
\put(5,5){\circle{10}}
\put(5,7.5){\line(0,0){5}}
\end{picture}
\end{minipage}\hspace*{1.5pt}
}

\newcommand{\Circ}{
\hspace*{1pt}\begin{minipage}{10pt}
\begin{picture}(10,10)
\put(5,5){\circle{10}}
\end{picture}
\end{minipage}\hspace*{1.5pt}
}
\newcommand{\loopbarU}{\hspace*{0.5pt}\begin{minipage}{8pt}
\begin{picture}(8,8)
\put(4,4){\circle{8}}
\put(4,0){\circle*{3}}
\put(4,6){\line(0,0){4}}
\end{picture}
\end{minipage}\hspace*{1pt}
}
\markboth{J. Uhing}{}
\title{{\bf\Large A Polynomial Invariant Of Twisted Graph Diagrams}}
\author{Jason Uhing}
\maketitle
\begin{abstract}
\noindent Twisted graph diagrams  are virtual graph diagrams with bars on edges. 
A bijection between abstract graph diagrams and {\tgd s} is constructed. Then a polynomial invariant of Yamada-type is 
developed which provides a lower bound for the virtual crossing number of virtual graph diagrams.
\end{abstract}
\\
\ccode{05C10, 57M27}
\section{Introduction} 
Let $G$ be a finite graph considered as a topological space. An embedding of $G$ into three-dimensional space is called
a {\it spatial graph}. A {\it regular projection} of $G$ onto a surface $S$ is a continious map $G\rightarrow S$ whose
multiple points are finitely many transverse double points away from the vertices of $G$. The image of $G$ under 
a regular projection together with over/under information given to the double points is called a 
{\it (regular) graph diagram}
on $S$. In \cite{flemmell01} regular graph diagrams are extended to virtual (regular) graph diagrams motivated by L.~Kauffman's 
theory of virtual links, see \cite{math.GT/0502014}.  A one-to-one correspondence between 
virtual links and so called 
abstract link diagrams is presented in \cite{kamakama01}. In the first part of this note the notion of an abstract link diagram 
is extended to an abstract graph diagram. Differently from \cite{kamakama01} we allow the disk/band surfaces to be non-orientable. This 
enables us to construct a bijection from {\agd s} to so called {\tgd s}. These diagrams are generalisations of virtual graph diagrams by adding 
bars to edges. Geometrically a bar corresponds to a twist of a band of the surface.  Concerning links this idea can be found in \cite{bour06}.

In chapter \ref{pure tgd} we interpret the polynomial of B.~Bollob\'as and O.~Riordan which is defined for possibly non-orientable disk/band 
surfaces, see \cite{bollrior02}, as a polynomial for pure twisted graph 
diagrams via their abstract graph diagrams. This leads to a polynomial invariant for twisted graph diagrams. The definition is similar to that of the 
Yamada polynomial in \cite{yama:89}. As an application  we obtain  a lower bound for the virtual crossing number of a virtual graph diagram.

\section{Abstract Graph Diagrams}
In this paper the underlying graph of a regular graph diagram may have several components. In addition, components 
without vertices, so called {\sf circle components}, are allowed. 
\begin{definition}
A pair $\left(S,D\right)$ is called an {\sf abstract graph diagram} if $S$ is an two-dimensional
disk/band surface, $D$ is a regular graph diagram on $S$ and (as a subset of $S$) 
a strong deformation retract of $S$. 
\end{definition}
The crossings and the vertices of an {\agd} are contained in the disks of the surface. Two examples 
for orientable surfaces are shown in figure 1.
\begin{center}
\scalebox{0.15}{\includegraphics{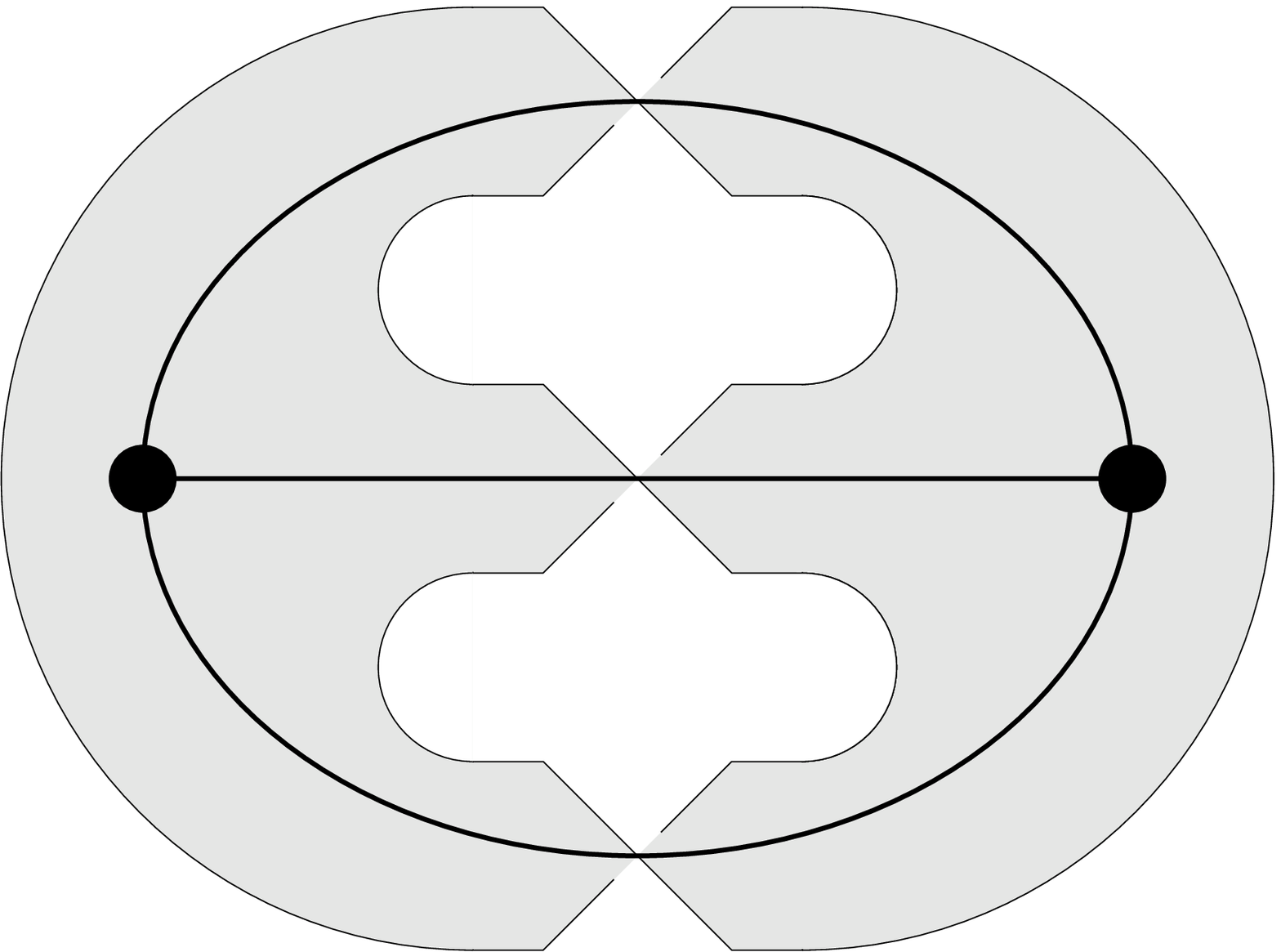}}\hspace*{2cm}\scalebox{0.35}{\includegraphics{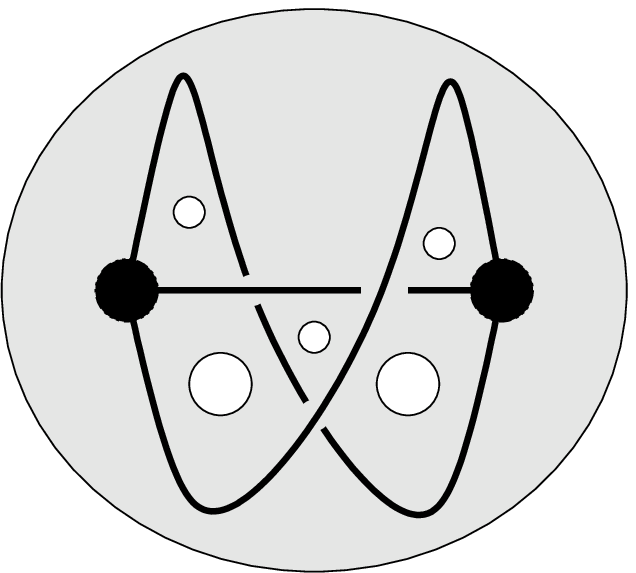}}\\
Figure 1
\end{center}
\begin{definition}
An {\agd} $\left(S,D\right)$ is obtained from another {\agd} $\left(S',D'\right)$ by an {\sf abstract Reidemeister move} of type I, II, III,
IV, V or VI if there exist 
embeddings $f : S\rightarrow F$, $f' : S'\rightarrow F$ for a closed  surface $F$, so 
that $f\left(D\right)$ is obtained from $f'\left(D'\right)$ by a Reidemeister move resp.~of  type  I to VI on $F$.
\end{definition}
Reidemeister moves are shown in \cite{flemmell01}, figure 2.
\begin{definition}
Two {\agd s} are said to be {\sf abstract Reidemeister move equivalent} 
or {\sf equivalent} if one is transformed into the other by a 
finite sequence of abstract Reidemeister moves.
\end{definition}
We denote the set of {\agd s} by $\cal{AG}$ and the corresponding set of equivalence classes by AG.
\section{Twisted  Graph Diagrams}
Extending classical graph diagrams by virtual crossings and virtual Reidemeister moves I$^*$ to V$^*$ we get 
virtual graph diagrams and virtual graphs. For  definitions see \cite{flemmell01}, chapter 2 and figure 4. 

We denote the set of virtual graph diagrams by $\cal{VG}$. The set of equivalence classes of $\cal{VG}$ generated
by Reidemeister moves I to VI and virtual Reidemeister moves I$^*$ to V$^*$ is denoted by VG. Following
\cite{bour06} we define {\it twisted graph diagrams} as virtual graph diagrams with bars on edges. 
The set of {\tgd s} is denoted by $\cal{TG}$. The set of equivalence classes generated by Reidemeister moves  I to VI , I$^*$ to V$^*$ and the 
twisted moves T1, T2, T3 and T4 of figure 2 is called TG. 
\begin{center}
\begin{minipage}{40pt}
\begin{picture}(40,40)
\qbezier(0,20)(0,20)(40,20)
\qbezier(20,0)(20,0)(20,40)
\put(17,33){\line(1,0){6}}
\end{picture}
\end{minipage}
=
\begin{minipage}{40pt}
\begin{picture}(40,40)
\qbezier(0,20)(0,20)(40,20)
\qbezier(20,0)(20,0)(20,40)
\put(17,7){\line(1,0){6}}
\end{picture}
\end{minipage}
\hspace*{30pt}
\begin{minipage}{12pt}
\begin{picture}(12,40)
\qbezier(6,0)(6,0)(6,40)
\put(3,15){\line(1,0){6}}
\put(3,25){\line(1,0){6}}
\end{picture}
\end{minipage}
=
\begin{minipage}{12pt}
\begin{picture}(12,40)
\qbezier(6,0)(6,0)(6,40)
\end{picture}
\end{minipage}
\hspace*{30pt}
\begin{minipage}{40pt}
\begin{picture}(40,40)
\qbezier(0,0)(3,3)(40,40)
\qbezier(0,40),(0,40)(17,23)
\qbezier(23,17)(23,17)(40,0)
\end{picture}
\end{minipage}
=
\begin{minipage}{70pt}
\begin{picture}(70,40)
\qbezier(0,40)(20,0)(35,20)
\qbezier(35,20)(50,40)(70,0)
\qbezier(0,0)(20,40)(33,22)
\qbezier(37,18)(50,0)(70,40)
\qbezier(66,2)(66,2)(70,5)
\qbezier(66,38)(66,38)(70,35)
\qbezier(4,2)(4,2)(0,5)
\qbezier(4,38)(4,38)(0,35)
\end{picture}
\end{minipage}\\[1ex]
\begin{minipage}{255pt}
T1 \hspace*{85pt} T2 \hspace*{83pt} T3
\end{minipage}\\[2ex]
\begin{minipage}{40pt}
\begin{picture}(40,40)
\qbezier(0,3)(25,40)(25,40)
\qbezier(13,0)(25,40)(25,40)
\qbezier(40,0)(25,40)(25,40)
\put(22,10){\circle*{1}}
\put(24,10){\circle*{1}}
\put(26,10){\circle*{1}}
\put(28,10){\circle*{1}}
\qbezier(2,10)(2,10)(6,7)
\qbezier(12,5)(12,5)(17,4)
\qbezier(36,4)(36,4)(40,6)
\put(25,39){\circle*{5}}
\end{picture}
\end{minipage}
=
\begin{minipage}{40pt}
\begin{picture}(40,40)
\qbezier(0,3)(0,10)(25,23)
\qbezier(25,23)(50,35)(25,40)
\qbezier(13,0)(13,7)(25,15)
\qbezier(25,15)(40,25)(25,40)
\qbezier(40,0)(40,7)(20,25)
\qbezier(20,25)(7,36)(25,40)
\put(22,6){\circle*{1}}
\put(24,6){\circle*{1}}
\put(26,6){\circle*{1}}
\put(28,6){\circle*{1}}
\put(22,30){\circle*{1}}
\put(24,30){\circle*{1}}
\put(26,30){\circle*{1}}
\put(28,30){\circle*{1}}
\put(25,39){\circle*{5}}
\end{picture}
\end{minipage}\\[1ex] T4 \\[2ex] Figure 2
\end{center}
 \section{Abstract vs. Twisted Graph Diagrams}
As in \cite{kamakama01} we define a map $\phi : \cal{TG}\rightarrow\cal{AG}$. In our setting, for a {\tgd} $E$ we have 
2-disks as regular neighborhoods 
for the crossings 
and the vertices. In figure 3 it is shown how the classical resp.~virtual crossings are replaced by a surface $S\subset \mathbb{R}^3$ 
and  a diagram $D$ on $S$. 
\begin{enumerate}
\item classical crossing
\raisebox{-40pt}{\hspace*{30pt}
\begin{picture}(40,40)
\qbezier(5,20)(5,20)(35,20)
\qbezier(20,5)(20,5)(20,17)
\qbezier(20,35)(20,35)(20,23)
\end{picture}
\raisebox{18pt}{\hspace*{15pt}$\longmapsto$\hspace*{15pt}}
\begin{picture}(40,40)
\qbezier(0,20)(0,20)(40,20)
\qbezier(20,0)(20,0)(20,17)
\qbezier(20,40)(20,40)(20,23)
\qbezier(0,15)(15,15)(15,0)
\qbezier(25,0)(25,15)(40,15)
\qbezier(25,40)(25,25)(40,25)
\qbezier(0,25)(15,25)(15,40)
\end{picture}
}
\item\label{virtual} virtual crossing
\raisebox{-40pt}{\hspace*{30pt}
\begin{picture}(40,40)
\qbezier(5,20)(5,20)(35,20)
\qbezier(20,5)(20,5)(20,35)
\end{picture}
\raisebox{18pt}{\hspace*{15pt}$\longmapsto$\hspace*{15pt}}
\raisebox{-5pt}{\scalebox{0.7}{\includegraphics{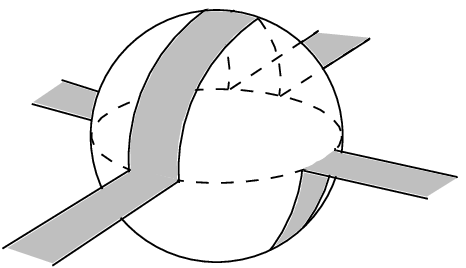}}}
}
\item vertex
\raisebox{-32pt}{\hspace*{50pt}
\begin{picture}(40,40)
\qbezier(20,20)(20,20)(10,0)
\qbezier(20,20)(20,20)(40,5)
\qbezier(20,20)(20,20)(17,40)
\put(20,20){\circle*{4}}
\end{picture}
\raisebox{20pt}{\hspace*{15pt}$\longmapsto$\hspace*{15pt}}
\begin{picture}(40,40)
\qbezier(20,20)(20,20)(10,0)
\qbezier(20,20)(20,20)(40,5)
\qbezier(20,20)(20,20)(17,40)
\qbezier(5,0)(17,20)(12,40)
\qbezier(15,0)(20,18)(38,2)
\qbezier(22,40)(23,20)(40,10)
\put(20,20){\circle*{4}}
\end{picture}}
\item\label{bar} bar
\raisebox{-20pt}{\hspace*{30pt}
\begin{picture}(40,40)
\qbezier(20,22)(20,20)(20,18)
\qbezier(5,20)(20,20)(35,20)
\end{picture}
\raisebox{17pt}{\hspace*{15pt}$\longmapsto$\hspace*{15pt}}
\raisebox{10pt}{\scalebox{0.5}{\includegraphics{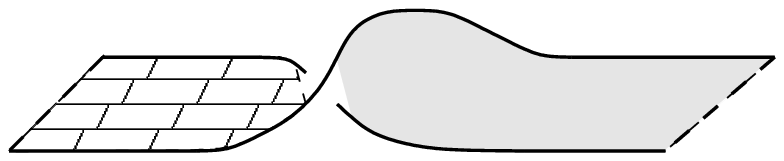}}}
}
\end{enumerate}\begin{center}
Figure 3
\end{center}
Note that up to homeomorphism in \ref{virtual}.~and \ref{bar}.~the surface does not depend on the sign of the crossing of the bands resp.~the twist. We define $\phi\left(E\right):=\left(S,D\right)$.
\begin{theorem}
The map $\Phi : TG \rightarrow AG$ defined by $\Phi\left(\left[E\right]\right):=\left[\phi\left(E\right)\right]$ is a bijection.
\end{theorem}
Before we give a proof of the theorem we construct a map $\psi : \cal{AG} \rightarrow$ TG and define $\Psi :$ AG 
$\rightarrow$ TG to be $\Psi\left(\left[\left(S,D\right)\right]\right):=\psi\left(\left(S,D\right)\right)$.

We remind the reader of the following notion from \cite{yasuhara:96}: Let $P\subset \mathbb{R}^3$ be a plane and $p:\mathbb{R}^3\rightarrow P$
a projection. The projection $p$ is {\it regular} for a disk/band surface $S\subset\mathbb{R}^3$ if the following conditions are satisfied:
\begin{enumerate}
\item For each $y\in p\left(S\right)$, $p^{-1}\left(y\right)\cap S$ consists of either one, two or infinitely many points.
\item\label{cond2} If $p^{-1}(y) \cap S$ consists of two points, then there are  two band parts $B_i, B_j$ of $S$ with 
	$y\in p\left(B_i\right)\cap p\left(B_j\right)$ such that  $p\left(B_i\right)$ and $p\left(B_j\right)$ meet as in figure 4. 
\item\label{cond3} If $p^{-1}\left(y\right) \cap S$ consists of infinitely many points, then there is exactly one band part 
	$B$ of $S$ with $y\in p\left(B\right)$ such that $p\left(B\right)$ is as in figure 5. 
\end{enumerate}
\begin{center}
\scalebox{0.7}{\includegraphics{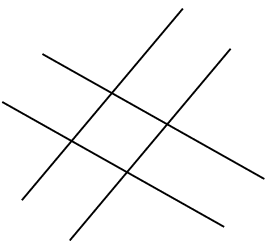}}\hspace*{3cm}\raisebox{17pt}{\scalebox{0.7}{\includegraphics{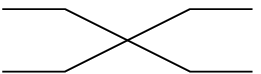}}} 
\\[1ex] Figure 4\hspace*{100pt}Figure 5
\end{center}
Let $\left(S,D\right)\in\cal{AG}$, $g: S\rightarrow \R^3$ an embedding and $p$ a regular projection for the disk/band surface 
$g\left(S\right)$. Consider $p\circ g\left(S\right)$ as a virtual graph diagram as follows: those double points of $p\circ g(D)$
belonging to the images of crossings of $D$ on $S$ are labelled with the corresponding over/under information. 
The remaining double points are considered as virtual crossings. 
Now we define a {\tgd} $E$ by adding a bar for every singularity like figure 5 coming from the image of $S$ under $p\circ g$. 
Then we set $\psi\left(\left(S,D\right)\right):=\left[E\right]$.
In the following propositions \ref{well phi} to \ref{well Psi} it is shown that  the maps $\phi$, $\Phi$, $\psi$ and $\Psi$ are well-defined.
\begin{proposition}\label{well phi}
$\phi$ is well-defined.
\end{proposition}
\begin{proof} By construction we have nothing to prove.
\end{proof}
\begin{proposition}\label{Phi}
$\Phi$ is well-defined.
\end{proposition}
\newcommand{\vgs}{\mbox{VG}'}
\begin{proofX}
Let $D,E \in \cal{TG}$. We have to show that $\phi(D)$ is equivalent to $\phi(E)$ as abstract graph diagrams
for $[D]=[E] \in$ TG. Suppose $D$ and $E$ differ by Reidemeister move VI. 
Thus they are identical outside a 2-disk $\Sigma\subset\mathbb{R}^2$. Abstract graph diagrams $\left(S_D,G_D\right)$ 
and $\left(S_E,G_E\right)$ embedded in three-dimensional space and being identical outside $\Sigma$ can be 
constructed. This is indicated in figure 6. As $S_D\cup\Sigma$ is homeomorphic to $S_E\cup\Sigma$ they are contained in 
a closed  surface constructed by glueing 2-disks to their boundary components. By definition of 
$\phi$ we have $\phi(D)=\left(S_D,G_D\right)$ and $\phi(E)=\left(S_E,G_E\right)$ since {\agd s} are considered
up to homeomorphism. Hence $\phi(D)$ is obtained from 
$\phi(E)$ by an abstract Reidemeister move. 

\begin{center}
\resizebox{2cm}{1.8cm}{\includegraphics{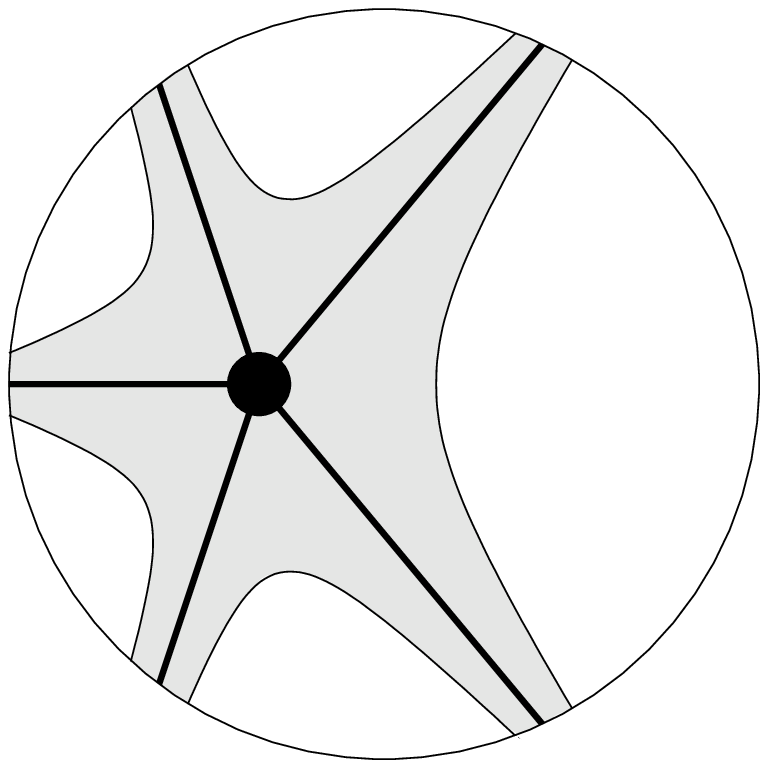}}\hspace*{2cm}
\resizebox{2cm}{1.8cm}{\includegraphics{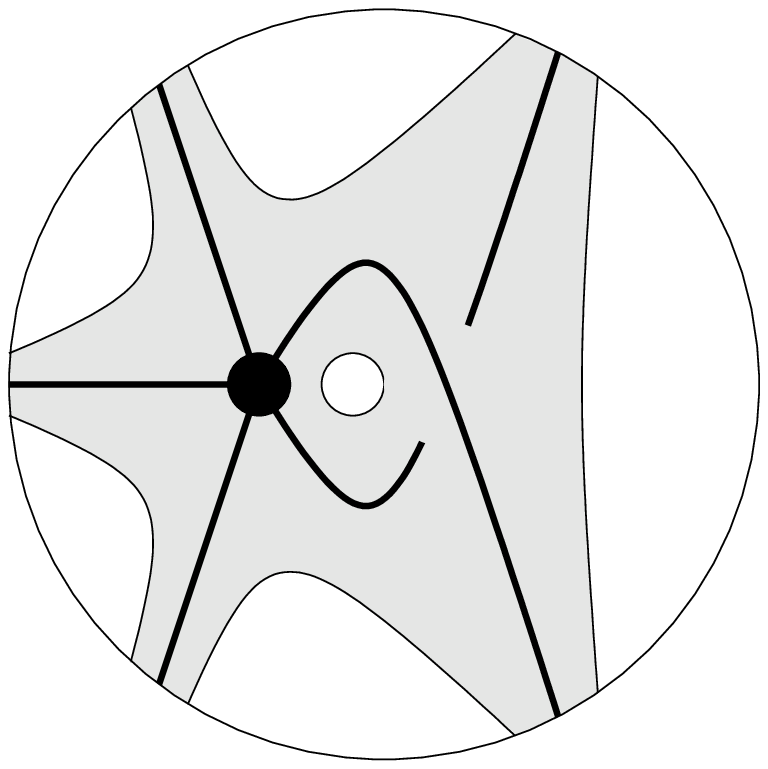}}\\
Figure 6
\end{center}

The remaining Reidemeister moves I, II, III, IV and V can be treated in an analogue manner.

Now suppose $D$ and $E$ differ by Reidemeister move IV$^*$. It is shown in figure 7 how the abstract graph 
diagrams can be obtained with respect to the disk $\Sigma$. There are several  possible ways
to choose the over/under behaviour of the bands inside a suitable neighborhood of the  disk, but this does not 
affect the type of the surface up to homeomorphism. Thus $\phi(D)\approx\phi(E)$, i.e. 
$\left[\phi(D)\right]=\left[\phi(E)\right]\in$ AG.

\begin{center}
\scalebox{0.25}{\includegraphics{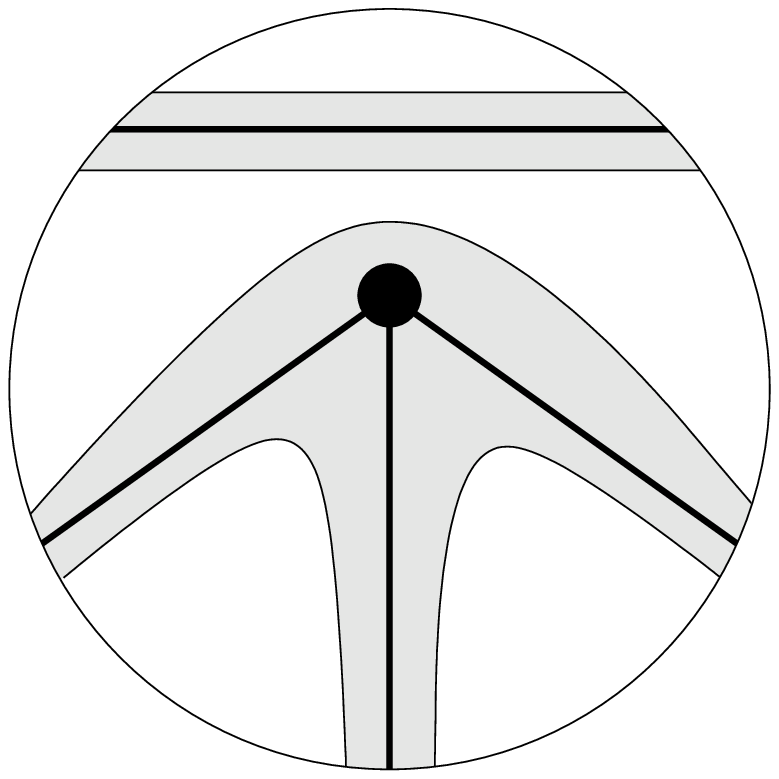}}\hspace*{1,5cm}
\scalebox{0.25}{\includegraphics{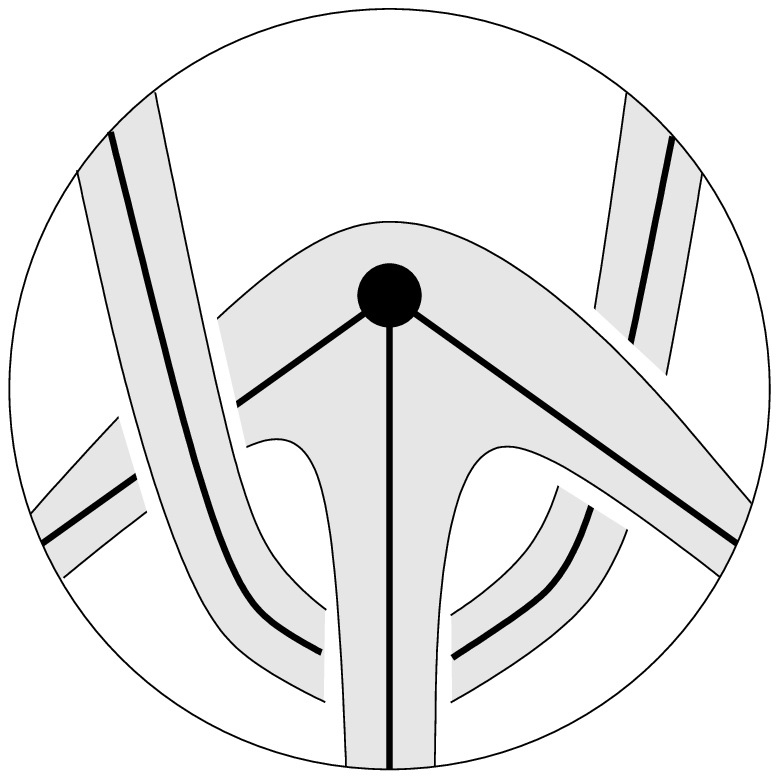}}\\
Figure 7
\end{center}

The remaining Reidemeister moves I$^*$, II$^*$, III$^*$ and V$^*$ can be treated in an analogue manner.

Now suppose $D$ and $E$ differ by a twisted move T2 inside the disk $\Sigma$. Obviously the correspondig 
{\agd s} are homeomorphic by the definition of $\phi$ in \ref{bar}., since two half-twists either cancel or become a full-twist. 
If $D$ and $E$ differ by T1, we argue just as in the case of pure virtual moves: one possible result of constructing 
the {\agd s} is shown in figure 8.
\begin{center}
\scalebox{0.5}{\includegraphics{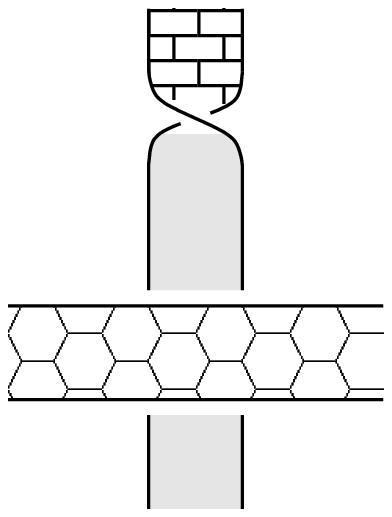}}\raisebox{40pt}{\hspace*{30pt}$\approx$\hspace*{30pt}}\scalebox{0.5}{\includegraphics{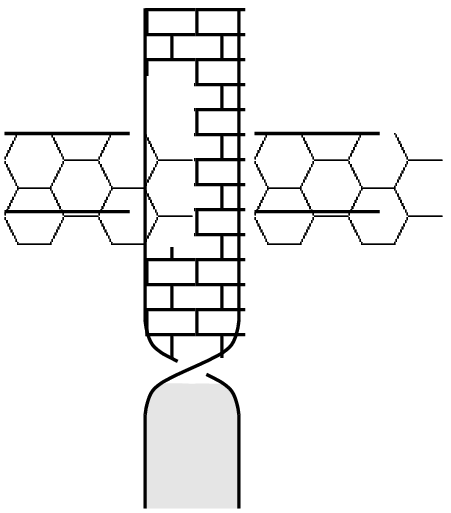}}\\ Figure 8
\end{center}
In figure 9 we see how a homeomorphism  may be obtained in the case of a T3-move. Rotate
 the surface around an horizontal axis and keep it fixed outside a suitable neighborhood of $\Sigma$. 
\begin{center}
\scalebox{0.6}{\includegraphics{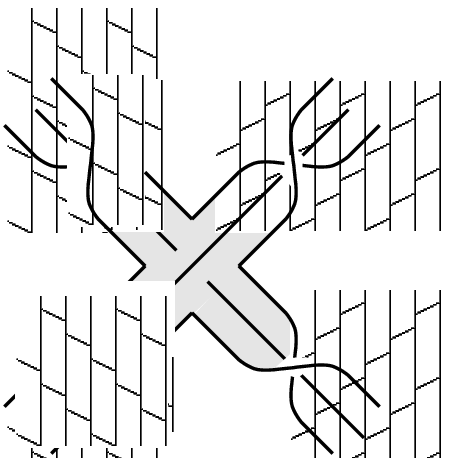}}\raisebox{20pt}{\hspace*{30pt}$\approx$\hspace*{30pt}}
\scalebox{0.6}{\includegraphics{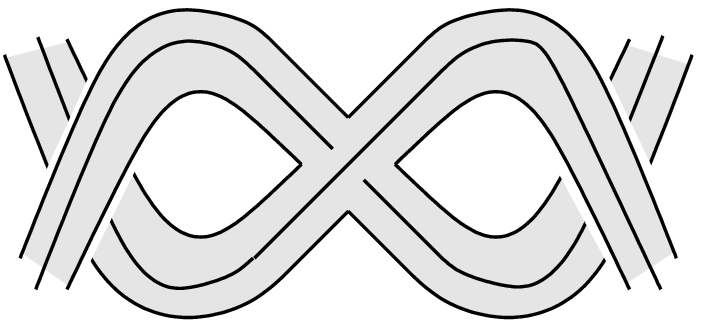}}\\ Figure 9
\end{center}
In the same way we treat the T4-move, i.e. $\hspace{-7pt}$ flipping the surface around an appropiate vertical axis.
\newline
\end{proofX}
\begin{remark} For Reidemeister move VI$^*$ of \cite{flemmell01}, figure 5 the proof of Propo- 
\begin{center}
\scalebox{0.3}{\includegraphics{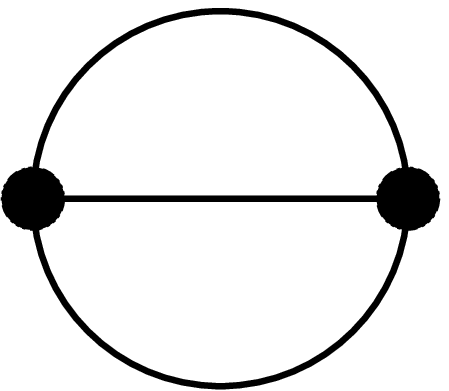}}\hspace*{10pt}\raisebox{13pt}{$\sim$}\hspace*{10pt}
\raisebox{-3pt}{\scalebox{0.3}{\includegraphics{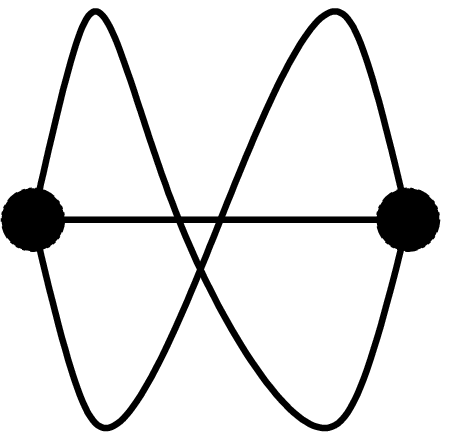}}}\hspace*{30pt}
\raisebox{-7pt}{
\scalebox{0.3}{\includegraphics{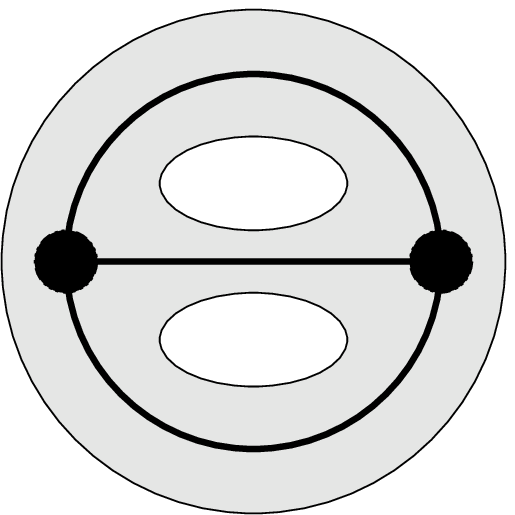}}\hspace*{10pt}\raisebox{22pt}{$\not\approx$}\hspace*{10pt}\scalebox{0.3}{\includegraphics{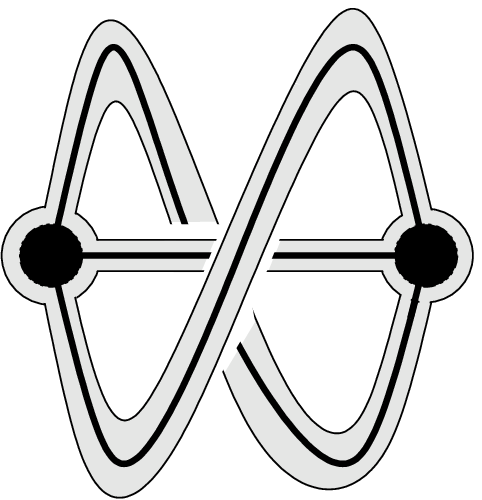}}}\\[1ex]
Figure 10
\end{center}
sition \ref{Phi}  does not work, because the 
corresponding surfaces may not be homeomorphic. An example is shown in figure 10.
\end{remark}
\begin{definition}\cite{kamakama01}
Let $D$ and $E$ be virtual graph diagrams of the same underlying graph such that they are identical inside
regular neighborhoods $N_1,\ldots, N_m$ of the crossings and the vertices. 
For $X\in \left\{D,E\right\}$ put  $W_X:=X\cap\overline{\left(\mathbb{R}^2\setminus 
\bigcup N_i\right)}$.
Then the  set $W_X$ is a union of immersed arcs. 
The diagrams $D$ and $E$ have the {\sf same Gauss data}, if there is a 1-1-correspondence between their immersed 
arcs $W_D$ and $W_E$ with respect to their boundary points in the union of the neighborhoods.   
\end{definition}
\begin{proposition}\label{gauss}
Two virtual graph diagrams  represent the same equivalence class in VG if they have the same Gauss data.
\end{proposition} 
\begin{proofX} As in the proof of Lemma 4.3 in \cite{kamakama01} the immersed arcs can be transformed 
into one another by a finite 
sequence of virtual Reidemeister moves up to isotopy. Aparently the forbidden move VI$^*$   is not required. 
\end{proofX} 
\begin{definition}\cite{bollrior02}\label{sign of band}
Let $S$ be a disk/band surface with an orientation chosen for every disk of $S$. Let $B$ be a band of $S$ with (possibly equal) 
incident disks $D_1$ and $D_2$. The {\sf sign} of $B$ is defined  to be $+1$ if the orientation  of $D_1$ is equal  to that of $D_2$ 
after moving  it along $B$. Otherwise it is defined to be $-1$. By an {\sf orientation of a band} we mean an orientation chosen for the 
topological disk belonging to the band.
\end{definition} 
\begin{proposition}\label{X}
Let $S$ be a disk/band surface with oriented disks and bands. For a band $B$ let $D_1, D_2$ be the incident
disks with $B\cap D_j=\partial B \cap \partial D_j=:I_j$ $\approx [0,1]$.
Then sign($B$) $=1$ if and only if  $\partial B$ induces  the same orientation  on $I_j$ as $\partial D_j$ for $j\in \{1,2\}$ or
$\partial B$ induces  the opposite orientation  on $I_j$ than $\partial D_j$ for $j\in \{1,2\}$.
\end{proposition} 
\begin{proofX} 
To show the if-part, 
suppose the orientation of $\partial B, \partial D_1$ and $\partial D_2$ correspond to each other as in figure 11.
\begin{center}
\begin{picture}(120,40)
\put(20,20){\circle{40}}
\put(99,20){\circle{40}}
\qbezier(39,27)(40,27)(80,27)
\qbezier(39,13)(40,13)(80,13)
\qbezier(60,27)(60,27)(63,29)
\qbezier(60,27)(60,27)(63,25)
\qbezier(63,13)(63,13)(60,15)
\qbezier(63,13)(63,13)(60,11)
\qbezier(0,20)(0,20)(-2,17)
\qbezier(0,20)(0,20)(2,17)
\qbezier(119,20)(119,20)(117,23)
\qbezier(119,20)(119,20)(121,23)
\put(10,20){{\footnotesize $D_1$}}
\put(95,10){{\footnotesize $D_2$}}
\put(50,17){{\footnotesize $B$}}
\end{picture}
\hspace*{30pt}
\begin{picture}(120,40)
\put(20,20){\circle{40}}
\put(99,20){\circle{40}}
\qbezier(39,27)(40,27)(80,27)
\qbezier(39,13)(40,13)(80,13)
\qbezier(60,27)(60,27)(63,29)
\qbezier(60,27)(60,27)(63,25)
\qbezier(63,13)(63,13)(60,15)
\qbezier(63,13)(63,13)(60,11)
\qbezier(0,20)(0,20)(-2,23)
\qbezier(0,20)(0,20)(2,23)
\qbezier(119,20)(119,20)(117,17)
\qbezier(119,20)(119,20)(121,17)
\put(10,20){{\footnotesize $D_1$}}
\put(95,10){{\footnotesize $D_2$}}
\put(50,17){{\footnotesize $B$}}
\end{picture}
\\[1ex] Figure 11 \hspace*{100pt} Figure 12
\end{center}
\noindent We conclude $\orient{D_2} = - \orient{B} =-\left( - \orient{D_1}\right) = \orient{D_1}$.
Now suppose the orientation of $\partial B, \partial D_1$ and $\partial D_2$ do not correspond to each other as in figure 12.
Then $\orient{D_2} = \orient{B} = \orient{D_1}$, i.e. ~$\hspace{-10pt}$ $\signt{B}=1$ in both  cases. 

We show the only-if-part in the same way  assuming the negation of the statement 
about the induced orientations in the proposition and conclude $\signt{B}=-1$.
\newline
\end{proofX}
\begin{definition}
A {\sf system of oriented disks (sod)} consists of the following data: Let $\cal{D}$ be a finite collection of oriented disks. 
Every disk $D\in \cal{D}$ comes with distinct points  $n_1,\ldots,n_k$ on $\partial D$ along the orientation of  
$\partial D$, where  $n_j\in \mathbb{Z}\setminus \{0\}$, see figure 13.
\begin{center}
\begin{picture}(40,40)
\qbezier(20,0)(28.29,0)(34.85,5.85)
\qbezier(34.85,5.85)(40,11.71)(40,20)
\qbezier(40,20)(40,28.29)(34.85,34.85)
\qbezier(34.85,34.85)(28.29,40)(20,40)
\qbezier[10](20,0)(11.71,0)(5.85,5.85)
\qbezier[10](5.85,5.85)(0,11.71)(0,20)
\qbezier(0,20)(0,28.29)(5.85,34.85)
\qbezier(5.85,34.85)(11.71,40)(20,40)
\put(20,40){\circle*{3}}
\put(34,35.5){\circle*{3}}
\put(27,1){\circle*{3}}
\put(3,30){\circle*{3}}
\put(18,45){\footnotesize $n_1$}
\put(36,39){\footnotesize $n_2$}
\put(26,-8){\footnotesize $n_3$}
\put(-12,30){\footnotesize $n_k$}
\put(13,13){\footnotesize $D$}
\qbezier(40,20)(40,20)(43,22)
\qbezier(40,20)(40,20)(37,22)
\end{picture}\\[1.5ex] Figure 13
\end{center}
In addition any number   appears exactly twice in $\cal{D}$.
A {\sf disk/band surface of a sod} is constructed by connecting each pair of equal numbers   by a band 
$B$ with $\signt{B}$ is the sign of the number.
\end{definition} 
\begin{proposition}\label{surface iso}
Let $\left(S,D\right)$ an {\agd}, $g : S\rightarrow \R^3$ an embedding, $p : \R^3 \rightarrow \R^2$ a regular projection for $g(S)$, 
$f:=p\circ g$ and $\mbox{pr} : \R^3\rightarrow \R^2$, $(x,y,z)\mapsto (x,y)$ the standard projection. Moreover let $E$ be the 
{\tgd} coming from the image of $S$ under $f$ and $\phi\left(E\right)=\left(S',D'\right)$ for some choice of $S'\subset\R^3$ according 
to the definition of $\phi$. 
Then there is an embedding $f_\phi : S \rightarrow \R^3$ such that $f_\phi\left(S\right)=S'$ and 
$$
\begin{diagram}
\node{S} \arrow{e,t}{g} \arrow{s,l}{f_\phi} \arrow{se,l}{f} 
\node{\R^3}\arrow{s,r}{p}\\
\node{\R^3} \arrow{e,b}{pr} \node{\R^2}
\end{diagram}
$$
is a commutative diagram. 
\end{proposition}
\begin{proof}
First, choose an orientation of the disks and the bands of $S$. Then there is an sod, such that $S$ is a disk/band surface of that sod. Via the 
orientation-preserving  homeomorphisms $f|_D : D \rightarrow f(D) \subset \R^2$ for every disk $D$ of $S$, we get  another sod  consisting  
of the disks $f(D)$ of the surface $S'$. The pairs of numbers on the boundaries of the disks define a one-to-one correspondence between the 
bands of $S$ and $S'$. It follows from Proposition \ref{X}, that those  corresponding bands have the same sign, as  $f$ preserves  the orientation  
of the boundaries of the disks and the bands. Therefore $S$ and $S'$ have to be homeomorphic, as they are  disk/band surfaces of  homeomorphic 
sod with the same signs on the bands. We conclude that $S'$ is an image of an embedding $f_\phi$ of $S$ into $\R^3$. 
From the definition of  $\phi$ it follows that the diagram comutes.
\end{proof}

\noindent To show, that $\psi$ is well-defined, we have 
\begin{proposition}\label{B}
Let $(S,D)$ be an {\agd}, $g, g' : S\rightarrow \mathbb{R}^3$ embeddings, $p, p' : \R^3\rightarrow \R^2$ regular projections for $g(S)$ resp.
$g'\left(S\right)$, $f:=p\circ g$ and $f':=p'\circ g'$.  Let $E$, $E'$ be the 
{\tgd s} coming from the image of $S$ under $f$ resp. $f'$. Then $E$ is equivalent to $E'$ in TG.
\end{proposition} 
\begin{proofX}
First, choose an orientation of the disks and the bands of $S$.
Suppose  $D$ is a disk of $S$ such that $f'(D)\subset \R^2$ has the opposite orientation of $f(D)\subset \R^2$. 
Depending on whether there is a real crossing or a vertex inside the disk we get a diagram $C$ equivalent to $E'$ by performing a T3- resp. a T4-move 
at $E'$ for {\it all} such disks. As in the proof of proposition \ref{Phi}, there is a homeomorphism $H : \R^3 \rightarrow \R^3$ coming from 
rotating that disks around $2\pi$ such 
that $H\left(\phi\left(E'\right)\right)=\phi(C)$. 
(To keep the notation short, by $\phi(\cdot)$ we mean only the surface-part of the {\agd}.)
As a result the disks $f(D)\subset\R^2$ of $\phi(E)$ and the disks $H\circ f'(D)\subset \R^2$ of $\phi(C)$ have 
the same orientation. Moreover, the diagrams  $E$ and $C$ have the same Gauss data. With $f_\phi$ and $f'_\phi$ being the embeddings introduced 
in proposition \ref{surface iso}, the composition
$$
\begin{diagram}
\node{h: \phi(E)} \node{S} \arrow{w,tb}{f_\phi}{\approx} \arrow{e,tb}{f'_\phi}{\approx} \node{\phi\left(E'\right)} \arrow{e,tb}{H}{\approx} \node{\phi(C)}
\end{diagram}
$$
maps the disks and bands  of $\phi(E)$ to the disks and bands of $\phi(C)$. It follows from Proposition \ref{X}, that the bands mapped onto each other
via $h$
have the same sign, because the disks have the same orientation. 
In the sense of definition \ref{sign of band}, i.e. moving an orientation  along the band, those bands must have the same number of twists modulo 2. 
Therefore  the corresponding arcs  of the diagrams  $E$ and $C$ have the same number of bars modulo 2. Combining this with Proposition 
\ref{gauss} and the twisted Reidemeister moves we see that $E\sim C$, thus $E\sim E'$.
\end{proofX}

\begin{proposition}\label{well Psi}
$\Psi$ is well-defined.
\end{proposition}
\begin{proofX}
We have to show $\Psi\left(\left[(S,D)\right]\right)=\Psi\left(\left[\left(S',D'\right)\right]\right)$ for equivalent abstract graph diagrams $(S,D)$ and $\left(S',D'\right)$. 
Assume that  $(S,D)$ and $\left(S',D'\right)$ differ by an abstract Reidemeister move. Then there are embeddings
$f : S\rightarrow F$ and $f' : S'\rightarrow F$ into a closed surface $F$, such that $f(D)$ and $f'\left(D'\right)$ 
differ by a Reidemeister move inside a disk $\Sigma \subset F$ of type I, II, III, IV, V or VI. Outside the disk the 
diagrams are identical, i.e. $f(D) \cap \overline{F\setminus \Sigma} = f'\left(D'\right)\cap\overline{F\setminus \Sigma}$.
Hence we may choose regular neighborhoods $N$ and $N'$ of $f(D)$ resp. $f'\left(D'\right)$ satisfying $N\approx f(S) \approx S$,
$N'\approx f'\left(S'\right)\approx S'$, $N\setminus \Sigma=N'\setminus \Sigma$ and $N\cup \Sigma=N'\cup \Sigma$. 
Applying $\psi$ to the abstract graph diagram $\left(N,f(D)\right)$ we get an  embedding $g:N\rightarrow \R^3$ and a regular projection $p$  for $g(N)$.
As $N$ and $N'$ are equal outside $\Sigma$ it is easy to construct an embedding $h : N\cup\Sigma=
N'\cup\Sigma \rightarrow \R^3$ and a projection $\tilde{p}$ regular for $h\left(N\cup\Sigma\right)$ 
with $h$ equal to $g$ when restricted to $N$, such that the {\tgd} $E$ belonging to $p\circ g\left(N\right)$ resp. $E'$ coming from
\newcommand{\h}{h\raisebox{-3pt}{$\mid \!\! N'$}}
$\tilde{p}\circ h\raisebox{-3pt}{$\mid \!\! N'$}\left(N'\right)$ differ by the same Reidemeister move mentioned above. Therefore we calculate
\begin{eqnarray*}
\Psi\left(\left[\left(S,D\right)\right]\right)&=&\Psi\left(\left[\left(N,f(D)\right)\right]\right) =\psi\left(N,f(D)\right)
	=\left[E\right]=\left[E'\right]\\
	&=&\psi\left(N',f'\left(D'\right)\right)=\Psi\left(\left[\left(N',f'\left(D'\right)\right)\right]\right)
	=\Psi\left(\left[\left(S',D'\right)\right]\right).
\end{eqnarray*}
\end{proofX}
 
\begin{proofofmaintheorem}
$\Phi$ injective: Let $D', E' \in \cal{TG}$, $\phi\left(D'\right)=\left(F_D,D\right)$, $\phi\left(E'\right)=\left(F_E, E\right)$ and  
\begin{eqnarray} 
\Phi\left(\left[D'\right]\right) &=& \Phi\left(\left[E'\right]\right). \label{voraus}
\end{eqnarray}
 The projection $pr : \R^3\rightarrow \R^2$, $(x,y,z)\mapsto (x,y)$ is regular for $F_E$ and $F_D$,  and $pr(D)=D'$, $pr(E)=E'$ by the definition 
of $\phi$. This implies $\left[D'\right]=\psi\left(\left(F_D,D\right)\right)$ and $\left[E'\right]=\psi\left(\left(F_E,E\right)\right)$. Thus
$$\left[D'\right]=\Psi\left(\left[\left(F_D,D\right)\right]\right)\stackrel{(\ref{voraus})}{=} \Psi\left(\left[\left(F_E,E\right)\right]\right)=\left[E'\right]$$
as $\Psi$ is well-defined.

To show that $\Phi$ is surjective let $(S,D) \in \cal{AG}$ and $\left[E\right]:=\Psi\left(\left[(S,D)\right]\right)$. Then $E$ is constructed via an embedding 
$g:S\rightarrow \R^3$ and a regular projection $p$ for $g(S)$.  Because of proposition \ref{surface iso}, the disk/band surface of the {\agd} 
$\phi\left(E\right)$ is homeomorphic to $S$. As the over/under informations of $D$ on $S$ correspond to those of  $E$, we get 
$\phi(E) \approx\left(S,D\right)$ and from that $\Phi\left([E]\right)=\left[\phi(E)\right]=\left[(S,D)\right]$. 
\end{proofofmaintheorem}

\section{Pure Twisted Graph Diagrams}\label{pure tgd}
\begin{definition}\cite{miya06}
A {\tgd} $E$ is called {\sf pure} if it has only virtual crossings.
\end{definition}
\begin{definition}\cite{bollrior02}
Let $E$ be a {\tgd} without circle components (\cc), $\phi\left(E\right)=\left(S,D\right)$ the corresponding {\agd} and 
\begin{enumerate}
\item $k\left(S\right) :=$ \# connected components of $S$,
\item $n\left(S\right):=$ first betti-number of $S$,
\item $b\left(S\right):=$ \# boundary components of $S$,
\item $t\left(S\right):= 0$, if $S$ is orientable, otherwise $t\left(S\right):=1$.
\end{enumerate}
Define $M\left(\emptyset\right):=1$ and
$$M\left(E\right)\left(y,z,w\right):=(-1)^{k\left(S\right)}y^{n\left(S\right)}z^{k\left(S\right) -b\left(S\right)+n\left(S\right)}w^{t\left(S\right)}$$
as a polynomial in $\mathbb{Z}\left[y,z,w\right]$ modulo $\left(w^2-w\right)$. Let $F$ be a {\tgd} possibly with \cc~For the number of \cc~having
an odd number of bars we write $o(F)$, for those with no or an even number of bars $e(F)$. Then define $Q\left(\emptyset\right):=1$ and
$$Q\left(F\right)\left(y,z,w\right):=\left(-1-y\right)^{e(F)}\left(-1-yzw\right)^{o(F)}\sum_{E\subset F}M(E).$$
Here by $E\subset F$ we mean a  twisted graph (sub-)diagram  $E$ (of $F$) belonging to a spanning subgraph of $F$ ignoring the \cc
\end{definition}
\begin{remark} 
The polynomial $M$ is that of \cite{bollrior02} for $X=0$. As $\left(S,D\right)$ is defined up to homeomorphism, so are $M$ and $Q$.
\end{remark}
\begin{remark}\label{Qinvariant}
From the previous section we know that Reidemeister moves I*, II*, III* and IV* do not change the {\agd s}. Hence $Q$ is invariant under those moves.
\end{remark}
\begin{remark}\label{QinvariantT}
From the previous section we know that Reidemeister moves T1, T2, T3 and T4 do not change the {\agd s}. Hence $Q$ is invariant under those moves as well.
\end{remark}
\begin{example}\label{exampleQ}
\begin{enumerate}
\item For a vertex we calculate $Q(\bullet)=M(\bullet)=-1$.
\item For a pure {\tgd} $F$ without \cc~we have $Q\left(F\right)= \sum_{E\subset F} M(E)$.
\item $Q\left(\Loop\right)=M\left(\bullet\right)+M\left(\Loop\right)=-1-y=Q\left(\Circ\right)$.
\item $Q\left(\loopbar\right)=M\left(\bullet\right)+M\left(\loopbar\right)=-1-yzw=Q\left(\circbar\right)$.
\end{enumerate}
\end{example}
\begin{definition}
Let $E$ be a {\tgd} looking like figure 14 inside a disk.  We call the {\tgd} $E/e$ the {\sf contraction of} $E$ {\sf along a twisted edge} $e$ and define it to 
be identical with $E$ outside the disk and to look  like figure 15 inside the disk. 
\begin{center}
\begin{picture}(73,73)
\qbezier[10](10.6,10.6)(0,21.21)(0,36.21)
\qbezier[10](0,36.21)(0,51.21)(10.6,61.82)
\qbezier[10](10.6,61.82)(21.21,72.42)(36.21,72.42)
\qbezier[10](36.21,72.42)(51.21,72.42)(61.82,61.82)
\qbezier[10](61.82,61.82)(72.42,51.21)(72.42,36.21)
\qbezier[10](72.42,36.21)(72.42,21.21)(61.82,10.6)
\qbezier[10](61.82,10.6)(51.21,0)(36.21,0)
\qbezier[10](36.21,0)(21.21,0)(10.6,10.6)
\put(21.21,36.21){\line(1,0){30}}
\put(51.21,36.21){\line(2,3){14}}
\put(51.21,36.21){\line(2,-3){14}}
\put(51.21,36.21){\line(3,2){19}}
{\linethickness{0.5pt}
\qbezier[3](65,38)(65,35)(62,30)
}
\put(51.21,36.21){\circle*{4}}
\put(21.21,36.21){\circle*{4}}
\put(21.21,36.21){\line(-2,3){14}}
\put(21.21,36.21){\line(-2,-3){14}}
\put(21.21,36.21){\line(-3,2){19}}
{\linethickness{0.5pt}
\qbezier[3](7.42,38)(7.42,35)(10.42,30)
}
\put(36.21,39.21){\line(0,-1){6}}
\end{picture}
\hspace*{80pt}
\begin{picture}(73,73)
\qbezier[10](10.6,10.6)(0,21.21)(0,36.21)
\qbezier[10](0,36.21)(0,51.21)(10.6,61.82)
\qbezier[10](10.6,61.82)(21.21,72.42)(36.21,72.42)
\qbezier[10](36.21,72.42)(51.21,72.42)(61.82,61.82)
\qbezier[10](61.82,61.82)(72.42,51.21)(72.42,36.21)
\qbezier[10](72.42,36.21)(72.42,21.21)(61.82,10.6)
\qbezier[10](61.82,10.6)(51.21,0)(36.21,0)
\qbezier[10](36.21,0)(21.21,0)(10.6,10.6)
\put(21.21,36.21){\circle*{4}}
\put(21.21,36.21){\line(-2,3){14}}
\put(21.21,36.21){\line(-2,-3){14}}
\put(21.21,36.21){\line(-3,2){19}}
{\linethickness{0.5pt}
\qbezier[3](7.42,38)(7.42,35)(10.42,30)
}
\qbezier[200](21.21,36,21)(36.21,72.42)(61.82,10.6)
\qbezier[200](21.21,36,21)(36.21,0)(61.82,61.82)
\qbezier[200](21.21,36,21)(45,4)(65.82,55.82)
{\linethickness{0.5pt}
\qbezier[4](36.21,43)(36.21,35)(36.21,29)
}
\qbezier(56.82,56.82)(56.82,56.82)(61.82,53.82)
\qbezier(58.82,46.82)(58.82,46.82)(63.82,43.82)
\qbezier(55.82,16.82)(55.82,16.82)(60.82,18.82)
\end{picture}\\[1ex]
\mbox{Figure 14}\hspace*{106pt}\mbox{Figure 15}
\end{center}
\end{definition}
\begin{remark}
The deletion $E-e$ is defined in the usual way no matter if $e$ has a bar or not. 
That is, we omit the edge $e$ in the diagram.
If $e$ has no bar, then the contraction $E/e$ is the usual one as well. 
 \end{remark}
\begin{remark}\label{method}
Contracting along an arbitrary edge is always possible, because with Reidemeister moves 
IV and  IV*   a situation like figure 14 can be obtained. 

\end{remark}
\begin{remark}\label{nobother}
By definition, contracting along a twisted edge is the same as contracting along an ordinary edge after performing a T4-move.
\end{remark}
\begin{remark}\label{mpol}
Note that even though the disk/band surfaces of the {\agd s} belonging to $E$ and $E/e$ are homeomorphic, the {\agd s} are not, because 
their diagrams are different. Nevertheless we have
$M\left(E\right)=M\left(E/e\right)$.
\end{remark}
\begin{definition}
A {\tgd} $E$ in $\R^2$ is {\sf split} into subdiagrams $E_1$ and $E_2$ if there is a simple close curve in $\R^2\setminus E$ seperating $\R^2$ into a disk 
$\Sigma$ and  $\R^2\setminus \Sigma$ containing $E_1$ resp.~$E_2$. We write $E=E_1\sqcup E_2$. 

If $E_1$ and $E_2$ share 
exactly one vertex $v$, $E$ is a union of $E_1$ and $E_2$ and there is a simple closed curve in $\left(\R^2\setminus E\right) \cup v$ meeting $v$ 
and seperating   $\R^2$ into a disk 
$\Sigma$ and  $\R^2\setminus \Sigma$ with $E_1\subset \Sigma$, $E_2\subset \left(\R^2\setminus\Sigma\right)\cup v$, we call $E$
a {\sf vertex connected sum} and name it $E=E_1\vee E_2$. 

An edge $e$ of $E$ is a {\sf cut-edge} if $E-e$ is a split diagram.
\end{definition}
\begin{proposition}\label{contrdel}
Let $E$ be a pure {\tgd} and $e$ a non-loop edge which is not a \cc~Then $Q\left(E\right) = Q\left(E/e\right) + Q\left(E-e\right)$.
\end{proposition}
\begin{proofX}
\newcommand{\ce}[1]{\left(-1-y\right)^{e\left(#1\right)}}\newcommand{\co}[1]{\left(-1-yzw\right)^{o\left(#1\right)}}
As $\ce E = \ce{E/e}=\ce{E-e}=:\alpha$ and $\co E = \co{E/e}=\co{E-e}=:\beta$ we calculate
\begin{eqnarray*}
Q\left(E\right) &=& \alpha\beta \sum_{F\subset E} M\left(F\right)=\alpha\beta\left[\sum_{\left\{F\subset E \mid e\notin F\right\}}M\left(F\right)+\sum_{\left\{F\subset E \mid e\in F\right\}}M\left(F\right)\right]\\
&\stackrel{\ref{mpol}}{=}& \alpha\beta\left[\sum_{F\subset E-e} M\left(F\right) + \sum_{F\subset E/e} M\left(F\right)\right]
= Q\left(E-e\right)+Q\left(E/e\right) .
\end{eqnarray*}
\end{proofX}

\begin{proposition}\label{split}
We obtain $Q\left(E_1 \sqcup E_2\right)=Q\left(E_1\right)Q\left(E_2\right)$ for pure {\tgd s} $E_1$ and $E_2$. 
\end{proposition}
\begin{proofX}
Before we proof the proposition we note that
\begin{equation}\label{first}
o\left(E_1\sqcup E_2\right)=o\left(E_1\right)+o\left(E_2\right), e\left(E_1\sqcup E_2\right)=e\left(E_1\right)+e\left(E_2\right) 
\end{equation}
for the \cc~, and there is a one-to-one correspondence between the sets
\begin{equation}\label{second}
\left\{F\subset E_1\sqcup E_2\right\} \longleftrightarrow \left\{F_1\subset E_1\right\} \times \left\{F_2\subset E_2\right\}.
\end{equation}
Moreover from \cite{bollrior02} we know that the proposition is true for the polynomial $M$. To abbreviate the notation let
$A:=-1-y$ and $B:=-1-yzw$ in the following calculation:
\begin{eqnarray*}
Q\left(E_1\right)Q\left(E_2\right) &=& \left[A^{e\left(E_1\right)}B^{o\left(E_1\right)}\sum_{F_1\subset E_1} M\left(F_1\right)\right]
							\left[A^{e\left(E_2\right)}B^{o\left(E_2\right)}\sum_{F_2\subset E_2} M\left(F_2\right)\right]\\
&\stackrel{\kref{first}}{=}& A^{e\left(E_1\sqcup E_2\right)}B^{o\left(E_1\sqcup E_2\right)}\sum_{F_1\subset E_1}\left[\sum_{F_2\subset E_2} M\left(F_2\right)\right] M\left(F_1\right)\\
&=&A^{e\left(E_1\sqcup E_2\right)}B^{o\left(E_1\sqcup E_2\right)}\sum_{F_1\times F_2 \subset \left\{F_1\subset E_1\right\} \times \left\{F_2\subset E_2\right\}} M\left(F_1\right)M\left(F_2\right)\\
&\stackrel{\kref{second}}{=}& A^{e\left(E_1\sqcup E_2\right)}B^{o\left(E_1\sqcup E_2\right)}\sum_{F\subset E_1 \sqcup E_2} M\left(F\right) \quad =\quad Q\left(E_1\sqcup E_2\right).
\end{eqnarray*}
\end{proofX}

\begin{proposition}\label{v-conn}
We have $Q\left(E_1 \vee E_2\right)=-Q\left(E_1\right)Q\left(E_2\right)$ for pure {\tgd s} $E_1$ and $E_2$. 
\end{proposition}
\begin{proofX}
First we note that there is a 1-1-correspondence between the sets
\begin{equation}\label{zwei}
\begin{diagram}
\node{\left\{F\subset E_1 \vee E_2\right\}} \arrow{e} \node{\left\{ F_1 \subset E_1\right\} \times \left\{ F_2 \subset E_2\right\}} \arrow{w}
\end{diagram}.
\end{equation}
For the number of \cc~we have
\begin{equation}\label{eins}
o\left(E_1\vee E_2\right)=o\left(E_1\right)+o\left(E_2\right), \quad e\left(E_1\vee E_2\right)=e\left(E_1\right)+e\left(E_2\right) .
\end{equation}
Moreover from \cite{bollrior02} we know that the proposition is true for the polynomial $M$. Using $A:=-1-y$ and $B:=-1-yzw$ we calculate
\begin{eqnarray*}
Q\left(E_1\right)Q\left(E_2\right)&=& \left[A^{e\left(E_1\right)}B^{o\left(E_1\right)}\sum_{F_1\subset E_1}M\left(F_1\right)\right]
					\left[A^{e\left(E_2\right)}B^{o\left(E_2\right)}\sum_{F_2\subset E_2}M\left(F_2\right)\right]\\
&\stackrel{\kref{eins}}{=}& A^{e\left(E_1\vee E_2\right)}B^{o\left(E_1\vee E_2\right)}  \sum_{F_1\subset E_1}\left[\sum_{F_2\subset E_2} M\left(F_2\right)\right]M\left(F_1\right)\\
&=& A^{e\left(E_1\vee E_2\right)}B^{o\left(E_1\vee E_2\right)}  \sum_{F_1\times F_2 \in \left\{ F_1 \subset E_1\right\} \times \left\{ F_2 \subset E_2\right\}}-M\left(F_1 \vee F_2\right)\\
&\stackrel{\kref{zwei}}{=}& A^{e\left(E_1\vee E_2\right)}B^{o\left(E_1\vee E_2\right)}  \sum_{F \subset E_1 \vee E_2} -M\left(F\right) \quad =\quad -Q\left(E_1\vee E_2\right).
\end{eqnarray*}
\newline
\end{proofX}
\begin{proposition}\label{cutedge}
If a pure {\tgd} $E$ has a cut edge $e$ then $Q\left(E\right)=0$.
\end{proposition}
\begin{proof}
We may write $E - e=E_1 \sqcup E_2$ and $E/e = E_1\vee E_2$ for appropriate subdiagrams $E_1$ and $E_2$. 
Note that we need not bother if $e$ has a bar or not because of remark \ref{nobother}. Thus
\begin{eqnarray*}
Q\left(E\right) &\stackrel{\ref{contrdel}}{=}& Q\left(E- e\right) + Q\left(E/e\right) = 
				Q\left(E_1\sqcup E_2\right) + Q\left(E_1\vee E_2\right) \\
		&\stackrel{\ref{split}, \ref{v-conn}}{=}& Q\left(E_1\right)Q\left(E_2\right) -  Q\left(E_1\right)Q\left(E_2\right) = 0.
\end{eqnarray*}
\end{proof}
The next proposition shows that $Q$ is a topological invariant in the sense that it does not care about vertices of degree 2.
\begin{proposition}
Let $E$ be a pure {\tgd} looking like figure 16 inside a disk $\Sigma$ and 
$E'$ the pure {\tgd} being identical with $E$ outside and looking 
like figure 17 inside $\Sigma$. Then $Q\left(E\right)=Q\left(E'\right)$.					
\begin{center}
\begin{picture}(40,40)
\qbezier[10](20,0)(28.29,0)(34.85,5.85)
\qbezier[10](34.85,5.85)(40,11.71)(40,20)
\qbezier[10](40,20)(40,28.29)(34.85,34.85)
\qbezier[10](34.85,34.85)(28.29,40)(20,40)
\qbezier[10](20,0)(11.71,0)(5.85,5.85)
\qbezier[10](5.85,5.85)(0,11.71)(0,20)
\qbezier[10](0,20)(0,28.29)(5.85,34.85)
\qbezier[10](5.85,34.85)(11.71,40)(20,40)
\qbezier(0,20)(0,20)(40,20)
\put(20,20){\circle*{3}}
\end{picture}
\hspace*{60pt}
\begin{picture}(40,40)
\qbezier[10](20,0)(28.29,0)(34.85,5.85)
\qbezier[10](34.85,5.85)(40,11.71)(40,20)
\qbezier[10](40,20)(40,28.29)(34.85,34.85)
\qbezier[10](34.85,34.85)(28.29,40)(20,40)
\qbezier[10](20,0)(11.71,0)(5.85,5.85)
\qbezier[10](5.85,5.85)(0,11.71)(0,20)
\qbezier[10](0,20)(0,28.29)(5.85,34.85)
\qbezier[10](5.85,34.85)(11.71,40)(20,40)
\qbezier(0,20)(0,20)(40,20)
\end{picture}
\\[1ex]
\mbox{Figure 16}\hspace*{60pt}\mbox{Figure 17}
\end{center}
\end{proposition}
\begin{proofX}
If the two segments of figure 16 belong to the same edge, we use example \ref{exampleQ} together with proposition \ref{split} 
to show the assertion. Now suppose those segments belong to different edges $e$ and $f$. Then $f$ is a cut edge for $E-e$ and 
$E'$ is equivalent to $E/e$ as $E$ is pure. Using the above propositions we calculate $Q\left(E\right)= Q\left(E-e\right) +Q\left(E/e\right) = 0 +Q\left(E'\right)$.
\end{proofX}

\section{An Invariant for Twisted Graph Diagrams} 
\newcommand{\s}[1]{\mathcal{S}\left(#1\right)}
\begin{definition}
Let $E$ be a {\tgd}. For a crossing $c$ of $E$ we define the {\sf spin} of $c$ to be $1, -1$ or $0$ as shown in figure 18. The pure {\tgd} obtained by replacing each
crossing with a spin is called a {\sf state} of $E$. The set of states will be denoted by $\s E$.  For $S\in \s E$ put $\left\{E\mid S\right\}:=a^{p-q}$, 
where $p$ and $q$ are the numbers of crossings with spin $+1$ and resp.~$-1$ in $S$. Now define a polynomial
\begin{equation}\label{R}
R\left(E\right)\left(a,z,w\right):=\sum_{S\in \s E} \left\{E\mid S\right\}Q\left(S\right)\left(-a-2-a^{-1}, z,w\right).
\end{equation}
\begin{center}
\poscrossO \hspace*{30pt} \nilcrossO \hspace*{30pt} \infcrossO \hspace*{30pt} \vertexO \\[2ex]
$c \hspace*{50pt} +1 \hspace*{60pt} -1 \hspace*{60pt} 0$ \\Figure 18
\end{center}
\end{definition}
\begin{remark}
If $E$ is pure we have $R\left(E\right)\left(a,z,w\right)=Q\left(E\right)\left(-a-2-a^{-1},z,w\right)$.
\end{remark}
\begin{proposition}\label{Rcondel}
The contraction/deletion formula is valid for the polynomial $R$, i.e.~$R\left(\edge\right)=R\left(\vertex\right)+R\left(\delete\right)$.
\end{proposition}
\begin{proofX} First we note 
\begin{equation}\label{statecondel}
\left\{\edge\mid S\right\}=\left\{\delete\mid S\right\}=\left\{\vertex\mid S\right\}
\end{equation}
for any state $S$. Hence we calculate
\begin{eqnarray*}
R\left(\edge\right) &=& \sum_{S\in \s{\edgeU}} \left\{ \edge\mid S\right\} Q\left(S\right) \stackrel{\ref{contrdel}}{=} 
	\sum_{S\in \s{\edgeU}} \left\{ \edge\mid S\right\} \left[Q\left(\delete\right)+Q\left(\vertex\right)\right]\\
&\stackrel{\kref{statecondel}}{=}& \sum_{S\in \s{\deleteU}} \left\{\delete\mid S\right\} Q\left(\delete\right) +
					\sum_{S\in \s{\vertexU}} \left\{\vertex\mid S\right\} Q\left(\vertex\right)\\
&=& R\left(\delete\right) + R\left(\vertex\right).
\end{eqnarray*}
\end{proofX}

\begin{proposition}\label{Rskein}
$R\left(\poscross\right)=aR\left(\nilcross\right)+ a^{-1}R\left(\infcross\right)+R\left(\vertex\right)$.
\end{proposition}
\begin{proofX}
Let $S$ be a state. We write $p=p\left(S,\cdot\right)$ and $q=q\left(S,\cdot\right)$. Then $p\left(S,\nilcross\right)=p\left(S,\poscross\right)-1$ and 
$q\left(S,\nilcross\right)=q\left(S,\poscross\right)$, hence $$\left\{\nilcross\mid S\right\}=a^{p\left(S,\nilcrossU\right)-q\left(S,\nilcrossU\right)}=
a^{p\left(S,\poscrossU\right)-1-q\left(S,\poscrossU\right)}=a^{-1}\left\{\poscross \mid S\right\}.$$ 
In an analogue manner we obtain $\left\{\infcross \mid S\right\}=a\left\{\poscross\mid S\right\}$ and 
$\left\{\vertex\mid S\right\}=\left\{\poscross\mid S\right\}$, therefore $R\left(\poscross\right) $
\begin{eqnarray*}
&=& \sum_{S\in\s{\nilcrossU}} \left\{\poscross\mid S\right\} Q(S) + \sum_{S\in\s{\infcrossU}} \left\{\poscross\mid S\right\} Q(S) +
				\sum_{S\in\s{\vertexU}} \left\{\poscross\mid S\right\} Q(S) \\
&=& aR\left(\nilcross\right)+ a^{-1}R\left(\infcross\right)+R\left(\vertex\right).
\end{eqnarray*}
\end{proofX}

\begin{proposition}\label{Rsplit}
We obtain $R\left(E_1 \sqcup E_2\right)=R\left(E_1\right)R\left(E_2\right)$ for {\tgd s} $E_1$ and $E_2$. 
\end{proposition}
\begin{proofX} 
Let $E=E_1\sqcup E_2$ and $S\in \s E$. Then $S=S_1\sqcup S_2$ for unique 
$S_i \in\s{E_i}$.
 We write $p=p\left(S,\cdot\right)$ and $q=q\left(S,\cdot\right)$.  Hence $p\left(E_1\sqcup E_2,S\right)=p\left(E_1,S_1\right) +p\left(E_2,S_2\right)$, $q\left(E_1\sqcup E_2,S\right)
=q\left(E_1,S_1\right) +q\left(E_2,S_2\right)$ and
therefore $\left\{E_1\sqcup E_2\mid S_1 \sqcup S_2\right\}=\left\{E_1\mid S_1\right\}\left\{E_2\mid S_2\right\}$. We check the equation as in
the proof of proposition \ref{split} using the assertion of that proposition.
\end{proofX}
\begin{proposition}\label{Rvcon}
We have $R\left(E_1 \vee E_2\right)=-R\left(E_1\right)R\left(E_2\right)$ for  {\tgd s} $E_1$ and $E_2$. 
\end{proposition}
\begin{proofX}
Replace $E_1\sqcup E_2$ with $E_1\vee E_2$ and \ref{split} with \ref{v-conn} in the proof of proposition \ref{Rsplit}.\newline
\end{proofX}
\begin{proposition}\label{Rcut}
$R\left(E\right)=0$ if a {\tgd} $E$ has a cut-edge.
\end{proposition}
\begin{proofX}
Let $e$ be a cut-edge of $E$, $E-e=E_1 \sqcup E_2, E_1 \subset \Sigma$ and $E_2 \subset \R^2\setminus \Sigma$. Then the components of  a state $S$  not containing  the arc $a$ of $S$ coming from the edge $e$ are either contained in $\Sigma$ or in $\R\setminus \Sigma$. Therefore $S -a$ is split, hence $a$ is a cut-edge for $S$.
Now the assertion follows by means of proposition   \ref{cutedge}.
\newline
\end{proofX}
\begin{example}\label{beispielR} Let $y=-a-2-a^{-1}$.
\begin{enumerate}
\item $R\left(\bullet\right)(a,z,w) = Q\left(\bullet\right)\left(-a-2-a^{-1},z,w\right)=-1$.
\item\label{sigma} $R\left(\Loop\right)(a,z,w)=Q\left(\Loop\right)\left(-a-2-a^{-1},z,w\right)=-1-\left(-a-2-a^{-1}\right)=a+1+a^{-1}=:\sigma=R\left(\Circ\right)$.
\item $R\left(\loopbar\right)=-1-\left(-a-2-a^{-1}\right)zw=-1+(\sigma+1)zw=R\left(\circbar\right)$.
\item \label{virtbouqet}$R\left(\doubleLoop\right)=Q\left(\doubleLoop\right)=M\left(\bullet\right)+2M\left(\Loop\right)+M\left(\doubleLoop\right)=-1-2y-y^2z^2$.
\item $R\left(\twoLoop\right)=-R\left(\Loop\right)R\left(\Loop\right)=-(-1-y)^2$.
\end{enumerate}
\end{example}
Because of propositions \ref{Rcondel}, \ref{Rskein}, \ref{Rsplit}, \ref{Rvcon}, \ref{Rcut} and example \ref{beispielR}.\ref{sigma} the propositions
 4 and 5 as well as theorem 5 of \cite{yama:89} are valid in our setting. We sum it up in 
\begin{proposition}\label{invariance RM}
The polynomial $R$ in \kref{R} is invariant under Reidemeister moves II, III, IV and up to multiplication with some $\left(-a\right)^{n}$ 
invariant under I and V. 
\end{proposition}  
\begin{proposition}\label{RiVirt}
The polynomial $R$ in \kref{R} is invariant under Reidemeister moves I*, II*, III* and IV*.
\end{proposition}
\begin{proofX}
If {\tgd s} $E$ and $E'$ differ by one of the moves mentioned in the assertion then for each state $S\in \s E$ there is a unique state 
$S' \in \s{E'}$ differing by the same Reidemeister move. As $E$ and $E'$ have the same crossings, we obtain 
$\left\{E\mid S\right\}=\left\{E'\mid S'\right\}$. From remark \ref{Qinvariant} we know $Q\left(S\right)=Q\left(S'\right)$. 
Thus the proof is finished by the definition of $R$.
\end{proofX}
\begin{proposition}\label{invariance vV}
The polynomial $R$ in \kref{R} is invariant under Reidemeister move V*.
\end{proposition}
\begin{proofX}
Because of propositions \ref{Rskein} and \ref{RiVirt} we may calculate $\RR{\VvirtL} =  a\RR{\nilcrossPa}+a^{-1}\RR{\infcrossPa} +\RR{\vertexPa} 
=a\RR{\nilcrossPb}+a^{-1}\RR{\infcrossPb} +\RR{\vertexPb} = \RR{\VvirtR}$.\newline
\end{proofX}
\begin{proposition}
The polynomial $R$ in \kref{R} is invariant under Reidemeister moves T1, T2, T4.
\end{proposition}
\begin{proofX}
The proof is exactly the same as in proposition \ref{RiVirt} exept for replacing remark \ref{Qinvariant} by remark \ref{QinvariantT}.\newline
\end{proofX}
\begin{proposition}
The polynomial $R$ in \kref{R} is invariant under Reidemeister move T3.
\end{proposition}
\begin{proofX} We calculate
\begin{eqnarray}
\RR{\Tiv} &=& a\RR{\TivInf}+a^{-1}\RR{\TivNil}+\RR{\TivVertex} \label{T4a} \\ 
&=& a\RR{\nilcross}+a^{-1}\RR{\infcross} +\RR{\vertex} \label{T4b} \\
&=& \RR{\poscross} \label{T4c}
\end{eqnarray}
using proposition \ref{Rskein} in \kref{T4a} resp.~\kref{T4c} and Reidemeister moves I*, II*, T2 and T4 in \kref{T4b}.
\newline
\end{proofX}
\section{Relations to other polynomials}
Let $S$ be a state. As usual we regard $S$ as a pure {\tgd} as well as the underlying abstract graph. 
Define $k(S) =$ \# components of $S$, $ n(S) =$ first betti number of $S$, $E(S) =$ set of edges of $S$, $V(S) =$ set of vertices of $S$, 
$u(S)=$ \# of circle components of $S$ and $\hat{F}=$ spanning subgraph/subdiagram of $S$ with edge set $F\subset E(S)$.

For a classical graph diagram $D$ the Yamada polynomial is defined to be
$$Y(D)(a)=\sum_{S\in\s D } \left\{D\mid S\right\} h(S)\left(-1,y\right), \quad y:=-a-2-a^{-1} \quad\mbox{where}$$
$$h(S)\left(-1,y\right)=\sum_{F\subset E(S)} (-1)^{k(S-F)}y^{n(S-F)}= \sum_{F \subset S} (-1)^{k(F)}y^{n(F)}.$$
Note that we consider each \cc~of S as a loop with one vertex of degree 2. In the last summation $F$ raises over all 
spanning subgraphs/subdiagrams of $S$.  
From \cite{yama:89} we know $h\left(\Loop\right)(-1,y)=-1-y$.
\begin{proposition}\label{RundY}
Let $D$ be a classical graph diagram possibly with circle components. Then $R(D)(a,1,1)=Y(D)(a)$.
\end{proposition}
\begin{proof}
For $z=w=1$ we get $R(D)(a,1,1)=\sum_{S\in \s D } \left\{D\mid S\right\} Q(S)(y,1,1)$ with $y=-a-2-a^{-1}$.  Thus we have to show $h(S)(-1,y)=Q(S)(y,1,1)$:
\begin{eqnarray}
Q(S)(y,1,1) &=& (-1-y)^{e(S)}(-1-y)^{o(S)}\sum_{E\subset S}M(E) (y,1,1)\nonumber\\
&=& (-1-y)^{u(S)}\sum_{E\subset S} (-1)^{k(E)} y^{n(E)} \nonumber \\
&=& h\left(\Loop\right)(-1,y)^{\# \loopU} h\left(S \setminus c.c.~\right)(-1,y) \label{loopCC}\\  &=& h(S) (-1,y). \nonumber
\end{eqnarray}
Note that in \kref{loopCC} we  identify each  \cc~of $D$ with a loop \Loop.
\newline
\end{proof}
Let $E$ be a virtual graph diagram possibly with \cc~In \cite{miya06} a polynomial is defined as follows:
$$H_E(a,1)=\sum_{S\in \s E} \left\{E \mid S\right\} Z_S\left(a+2+a^{-1}\right),$$
\begin{equation} Z_S(-y)=(-1-y)^{u(S)}y^{-\# V(S)} \sum_{F\subset E(S)} (-y)^{k\left(\hat{F}\right)} y^{\# E(F)}\label{Z}\end{equation}
where $y=-a-2-a^{-1}$. It turns out that this is the Yamada polynomial for virtual graphs introduced in \cite{flemmell01}. For the convienience of the reader 
we proof this fact in this context.
\begin{proposition}
Let $E$ be a virtual graph diagram possibly with circle components. Then $H_E(a,1)=R(E)(a,1,1)$.
\end{proposition}
\begin{proofX}
Set $y=-a-2-a^{-1}$.  It is sufficient to show $Z_S(-y)=h(S)(-1,y)$ for a state $S$ because of proposition \ref{RundY}. From \kref{Z} we get
\begin{eqnarray*}
Z_S(-y) &=& (-1-y)^{u(S)}\sum_{F\subset E(S)} (-1)^{k\left(\hat{F}\right)}y^{-\# V\left(\hat{F}\right) + k\left(\hat{F}\right)+ \# E\left(\hat{F}\right)} \\
&=& (-1-y)^{u(S)} \sum_{F\subset S} (-1)^{k(F)} y^{n(F)} \quad = \quad h(S)(-1,y).
\end{eqnarray*}
\end{proofX}\section{Applications}
By contrast with the Yamada polynomial our polynomial distinguishes certain diagrams. The reason is that the Yamada polynomial ignores the
virtual crossings of a virtual bouquet and the R-polynomial does not, see example \ref{beispielR}.\ref{virtbouqet}. The following two diagrams from 
\cite{flemmell01}, figure 20 have the same Yamada polynomial but different R-polynomials.
\begin{example}  Let $y=-a-2-a^{-1}$.
\newcommand{\rr}[1]{R\left(#1\right)}
\newcommand{\qq}[1]{Q\left(#1\right)}
\newcommand{\mm}[1]{M\left(#1\right)}
\begin{enumerate}
\item $\rr{\ThetaKlein}=\rr{\ThetaOhne}+\rr{\twoLoop}=\rr{\Circ}-\rr{\Circ}^2=-1-y-(-1-y)^2$
\item $\rr{\ThetaV}=\rr{\eight}+\rr{\doubleLoop}=\qq{\eight}+\qq{\doubleLoop}=\mm{\twovertex}+2\mm{\linie}+\mm{\eight}+\qq{\doubleLoop}=-2-3y-y^2z^2$.
\end{enumerate}
\end{example}
\begin{definition}\label{supgen}\cite{kamakama01}
Let $S$ be an orientable, connected  disk/band  surface. The minimum genus among all closed orientable surfaces in which $S$ is 
embeddable is called the {\sf supporting genus} of $S$.
\end{definition} 
\begin{remark}\label{glue}\cite{kamakama01}
Glueing 2-disks to the boundary components of $S$ in definition \ref{supgen} we obtain a closed orientable surface $\sigma(S)$ 
realizing the supporting genus of $S$.
\end{remark}
\begin{proposition}
Suppose  $S$ is  an orientable, connected disk/band surface. Then $n\left(\sigma(S)\right) = 2-\chi\left(\sigma(S)\right)=1-b(S)+n(S)$.
\end{proposition} 
\begin{proof} The surface $S$ is homeomorphic to a disk with  bands attached, see \cite{lick} figure 6.1. 
We write $n_1(S)$ for the number of generators of $H_1S$ coming from the {'handles'} and $n_2(S)$ for the number of generators
belonging to boundary components. Then $n(S)=n_1(S)+n_2(S)$, $n_2(S)=b(S)-1$ and $n\left(\sigma(S)\right)=n_1(S)=n(S)-n_2(S)=n(S)-b(S)+1$. As
$\sigma(S)$ is connected and orientable, the rank of $H_0S$ and $H_2S$ is $1$. The first equation follows immediatly.
\newline
\end{proof}
Consider the maximal degree of $z$ in the polynomial $R$ resp.~$Q$. Because of propositions \ref{invariance RM} to \ref{invariance vV} we will call it the $z$-degree.  
\begin{proposition}
For a classical graph diagram $D$ the $z$-degree of $R(D)$ is zero.
\end{proposition}
\begin{proof}
Each state $S$ of $D$ has neither virtual nor real crossings. Hence  $S$ is a planar embedding. Let $y=-a-2-a^{-1}$. As
$D$ is not twisted, we have $Q(S)(y,z,w)=(-1-y)^{u(S)}\sum_{E \subset S} (-1)^{k(F)}y^{n(F)}z^{k(F)+n(F)-b(F)}$ where $F$ denotes the surface-part of the 
{\agd} corresponding to the subdiagram $E$ of $S$. Let $F_i$ be the components of $F$. 
Then $k(F)-b(F)+n(F) = \sum k\left(F_i\right) -b\left(F_i\right) +n\left(F_i\right) = \sum 1-b\left(F_i\right) +n\left(F_i\right) = 
\sum 2- \chi\left(\sigma\left(F_i\right)\right)$. Each component of $S$ is a planar embedding,  hence $F_i$ is homeomorphic to a planar embedding 
of a disk/band surface. Thus $\sigma\left(F_i\right) \approx S^2$, i.e.~$\chi\left(\sigma\left(F_i\right)\right)=2$ finishing the proof.
\newline
\end{proof}
As an immediate corollary we have
\begin{proposition}
If the degree of $z$  in $R(D)$ is not zero, then $D$ is not a classical graph diagram.
\end{proposition}
Suppose $E$ is a virtual graph diagram. For the number of virtual crossings of $E$ we write $\#vcr(E)$.
\begin{definition}
The {\sf virtual crossing number} $vcr(E)$ of a virtual graph diagram $E\in\mathcal{VG}$ is defined to be 
$\min\left\{\#vcr\left(E'\right)\mid E'\sim E \;\mbox{in}\; \mathcal{VG}\right\}$.
\end{definition}  
\begin{proposition}
For a virtual graph diagram the $z$-degree of $R$ is bounded above by the virtual crossing number as follows: $z$-degree $R(E)\leq 2\, vcr(E)$.
\end{proposition}
\begin{proofX}
Let $E_1,\ldots,E_n$ be the components of $E$. Firstly suppose $E$ is pure. The surface $S$ of the {\agd} $\phi(E)$ consists of components $S_i$ 
coming from $\phi\left(E_i\right)$. Consider each $E_i$ as a diagram in $S^2$. Instead of modifying $S_i$ as in 2.~of figure 3 we add a handle and let 
the surface  $S_i$ pass it.  Then $S_i$ is embedded in a closed orientable surface $F_{g_i}$ of genus $\#vcr\left(E_i\right)=g_i$. In an analogue 
manner $S$ is embedded in $F_g$ where $g=\#vcr(E)$. Thus $\sum g_i = \sum \#vcr\left(E_i\right)\leq \#vcr(E)= g$. Now attach disks to the boundary 
componnets of $S_i$ to obtain $\sigma\left(S_i\right)$ having genus $\tilde{g_i}$. From remark \ref{glue} we know $\tilde{g_i}\leq g_i$. Then 
$z$-degree$M(E) = k(S)-b(S)+n(S)= \sum 1- b\left(S_i\right)+n\left(S_i\right)=\sum n\left(\sigma\left(S_i\right)\right)=\sum 2 \tilde{g_i}\leq 
\sum 2g_i \leq 2g$. As $E$ has no twists, $Q$ has the form $Q(E)(y,z,w)=(-1-y)^{u(S)}\sum_{D \subset E} M(D)(y,z,w)$. 
Therefore we calculate $z$-degree $Q(E) = \max\left\{z\mbox{-degree }M(D)\mid D\subset E\right\} \leq \max\left\{2\#vcr(D)\mid D\subset E\right\}\leq 
2\#vcr(E)$. 

Now suppose $E$ is a virtual graph diagram not necessarily pure. Abbreviating $y=-a-2-a^{-1}$ we have 
\begin{eqnarray*}
z\mbox{-degree } R(E)(a,z,w)&=& \max\left\{z\mbox{-degree } Q(S)(y,z,w)\mid S\in\mathcal{S}(E)\right\}\\ 
&\leq& \max\left\{2\#vcr(S)\mid S\in\mathcal{S}(E)\right\}\\ &\leq &2\#vcr(E).
\end{eqnarray*}
Taking the minimum over all diagrams equivalent to $E$ in $\mathcal{VG}$ finishes the proof.
\newline
\end{proofX}
\begin{example}\label{hc2}
Consider a virtual diagram $E$ of the handcuff graph shown in figure 19. \\[1ex]
$$
\begin{minipage}{160pt}
\begin{picture}(160,70)
{\linethickness{0.7pt}
\qbezier(0,35)(30,55)(50,35)
\qbezier(50,35)(65,15)(80,35)
\qbezier(80,35)(95,50)(110,35)
\qbezier(110,35)(120,25)(80,0)
\qbezier(80,0)(60,-15)(0,35)

\qbezier(160,35)(130,55)(113,38)
\qbezier(107,32)(95,15)(80,35)
\qbezier(80,35)(65,50)(53,38)
\qbezier(47,32)(40,25)(80,0)
\qbezier(80,0)(100,-15)(160,35)

\qbezier(0,35)(80,110)(160,35)

\put(0,35){\circle*{6}}
\put(160,35){\circle*{6}}}
\end{picture}
\end{minipage}
$$\\
\begin{center}Figure 19\end{center}
We use the algorithm \cite{uhingALGO:07} to determine the polynomial of $E$. The result is
$$R(E)=-a^{-5}\left(r_4(a)z^4+r_2(a)z^2+r_0(a)\right)$$ with
\begin{eqnarray*}
r_4(a)&=& a^9 + 8a^8 + 28a^7 + 56a^6 + 70a^5 + 56a^4 + 28a^3 + 8a^2 + a,\\
 r_2(a)&=& - 15a^8 - 43a^7 - 70a^6 - 81a^5 - 70a^4 - 37a^3 - 6a^2 - 2a^9 + 3a+1,\\
r_0(a)&=& 6a^8 + 14a^7 + 13a^6 + 11a^5 + 14a^4 + 10a^3 - a^2 - 3a.
\end{eqnarray*}
We conclude $4=z\mbox{-degree}\, R(E)\leq 2vcr(E)$. Thus $vcr(E)=2$.
\end{example}
\begin{example}\label{hc4}
Consider a diagram of the {\sf handcuff graph of order} 4 depicted below. Let $S$ be the state  
with zero spin at each crossing.  
$$
\begin{minipage}{200pt}
\begin{picture}(200,200)
{\linethickness{0,7pt}
\put(100,100){\circle*{6}}
\put(180,120){\circle*{6}}
\put(120,20){\circle*{6}}
\put(20,80){\circle*{6}}
\put(80,180){\circle*{6}}
\put(180,100){\circle{40}}
\put(100,20){\circle{40}}
\put(20,100){\circle{40}}
\put(100,180){\circle{40}}
\qbezier(100,100)(100,100)(155,100)
\qbezier(165,100)(220,100)(120,20)
\qbezier(100,100)(120,20)(85,20)
\qbezier(75,20)(20,20)(20,80)
\qbezier(100,100)(20,85)(20,115)
\qbezier(20,125)(20,180)(80,180)
\qbezier(100,100)(80,180)(115,180)
\qbezier(125,180)(180,180)(180,120)
}
\end{picture}
\end{minipage}
$$
\begin{center}Figure 20\end{center}
Then $k(S)=1, b(S)=1$ and $n(S)=8$.
We conclude 
$8=k(S)+n(S)-b(S)\leq z-\mbox{degree } R(E) \leq 2vcr(E)$. Hence $vcr(E)=4$.
\end{example}
\begin{proposition}
For every $n\in \mathbb{N}$ there is a virtual graph diagram with virtual crossing number $n$.
\end{proposition}
\begin{proofX}
For $n=1$ see example \ref{beispielR}.4, for $n=2$  
example \ref{hc2}. For $n\geq 3$ consider the handcuff graph of order $n$ as in example \ref{hc4}. 
\end{proofX}

\bibliography{Artikel}
\bibliographystyle{abbrv}
\end{document}